\documentclass[a4, 12pt]{amsart}
\UseRawInputEncoding
\usepackage{mathrsfs}
\usepackage{amsmath}
\usepackage{amsfonts}
\usepackage{amssymb}
\usepackage{amsmath,amssymb,amsthm}
\usepackage{amsmath,amssymb,amsthm,amscd}
\usepackage[frame,cmtip,arrow,matrix,line,graph,curve]{xy}
\usepackage{graphpap, color}
\usepackage[mathscr]{eucal}
\usepackage{xcolor}
\usepackage{verbatim}
\usepackage[colorlinks, linkcolor=red, anchorcolor=blue,citecolor=blue]{hyperref}
\usepackage{cite}

\numberwithin{equation}{section} \numberwithin{equation}{section}
\newtheorem{thm}{Theorem}[section]
\newtheorem{lem}[{thm}]{Lemma}

\newtheorem{prop}[{thm}]{Proposition}
\newtheorem{corr}[{thm}]{Corollary}
\newtheorem{rem}{Remark}[section]{\bfseries\upshape}

\newtheorem*{prob-a}{Problem A}
\newtheorem*{prob-b}{Problem B}
\newtheorem*{prob-c}{Problem C}

\newtheorem{con}[{thm}]{Conjecture}
\newtheorem*{ack}{Acknowledgment}

\newtheorem*{GPPW}{Genernal Payne-P\'{o}lya-Weinberger Conjecture}
\newtheorem*{PPW}{Therorem~(Payne-P\'{o}lya-Weinberger Conjecture)}

\newcommand{\DOI}[1]{doi: \href{https://doi.org/#1}{#1}}

\renewcommand{\oddsidemargin}{5mm}

\makeatletter

\title[Eigenvalues of Xin-Laplacian]{Eigenvalues of Xin-Laplacian on \\ Complete  Riemannian manifolds}
\author[L. Zeng \& Z. Zeng]{Lingzhong Zeng and Zhouyuan Zeng}

\address{Lingzhong Zeng
\\  \newline \indent School of Mathematics and Statistics
\\  \newline \indent   Jiangxi Normal University, Nanchang 330022,  China. lingzhongzeng@yeah.net \\ \newline \indent
Zhouyuan Zeng\\   \newline \indent School of Mathematics and Statistics\\
 \newline \indent  Jiangxi Normal University,
Nanchang 330022, China. zhouyuan0811@163.com}

\begin{document}
\maketitle

\begin{abstract}  In this paper,
we consider Dirichlet eigenvalue problem which is related to  Xin-Laplacian on the bounded domain of
complete Riemannian manifolds. By establishing the general formulas, combining with some results of Chen and Cheng type, we prove some eigenvalue inequalities. As some applications, we consider the eigenvalues on some Riemannian manifolds admitting with special functions, the  translating solitons, minimal submanifolds on the Euclidean spaces, submanifolds on the unit spheres, projective spaces and so on. In particular,  for translating solitons, our eigenvalue inequalities are universal. In addition, we investigate the closed eigenvalue problem for the Xin-Laplacian and generalize the Reilly's result on the first eigenvalue of the Beltrami-Laplacian.  As some remarkable applications, we obtain a very sharp estimate for the upper bound of the second nonzero eigenvalue(without counting multiplicities of eigenvalues) of the Beltrami-Laplacian on the minimal isoparametric hypersurfaces and focal submanifolds in the unit sphere, which leads to a conjecture and is the most fascinating part of this paper. Furthermore, our result hints $2n$ may be the second eigenvalue of Beltrami-Laplacian on the isoparametric hypersurfaces who are not isometric a unit sphere.
\end{abstract}


\section{Introduction}

Let $\mathcal{M}^{n}$ be an $n$-dimensional, complete Riemannian submanifold isometrically immersed into the $(n+p)$-dimensional Euclidean space $\mathbb{R}^{n+p}$.
Suppose that $g_{0}$ is the standard metric on the Euclidean space  $\mathbb{R}^{n+p}$ and $g$ is a Riemannian metric on $\mathcal{M}^{n}$ induced from the Euclidean space  $\mathbb{R}^{n+p}$.
Throughout this paper, we use $\langle\cdot,\cdot\rangle_{g}$, $|\cdot|_{g}^{2}$, ${\rm div}$, $\Delta$, $\nabla$ and $\nu^{\top}$ to denote the Riemannian inner product associated with the induced metric $g$, norm  with respect to the inner product $\langle\cdot,\cdot\rangle_{g}$,  divergence, Laplacian, the gradient operator on Riemannian manifolds  $\mathcal{M}^{n}$ and the projection of the vector $\nu$ onto the tangent bundle of $\mathcal{M}^{n}$, respectively. Moreover, we use $\langle\cdot,\cdot\rangle_{g_{0}}$, $|\cdot|_{g_{0}}^{2}$ and $\nu^{\bot}$ to represent the standard Euclidean inner product, norm on $\mathbb{R}^{n+p}$ and the projection of $\nu$ onto the normal bundle of $\mathcal{M}^{n}$, respectively. Next, we define Xin-Laplacian (or call it $\mathfrak{L}_{\nu}$ operator) as follows:

\begin{equation}\label{L-equ}  \mathfrak{L}_{\nu}(\cdot) =\Delta(\cdot)+ \langle\nu,\nabla(\cdot)\rangle_{g_{0}}=e^{-\langle\nu,X\rangle_{g_{0}}}{\rm div}(e^{\langle\nu,X\rangle_{g_{0}}}\nabla(\cdot)).\end{equation}  Xin-Laplacian is  an elliptic differential operator and introduced by Xin
in \cite{Xin2}. From the viewpoint of geometry, Xin-Laplacian plays  an
important role for the geometric understanding of the translating solitons, see \cite{CQ,Xin2}. We remark that this operator  is similar to the $\mathfrak{L}$ operator introduced by  Colding and Minicozzi in \cite{CM} and Witten-Laplacian given by
$\Delta_{f}(\cdot)=\Delta(\cdot)-\langle\nabla f,\nabla(\cdot)\rangle_{g}$, where $f$ is a potential function defined on $\mathcal{M}^{n}$(cf. \cite{FLL,FS,WW}).
It can be shown that the elliptic differential operator $\mathfrak{L}_{\nu}$ is a self-adjoint operator with
respect to the weighted measure $e^{\langle\nu,X\rangle_{g_{0}}}dv$,  namely, for any $u, w
\in C_{0}^{2}(\Omega)$,
\begin{equation}
\begin{aligned}
\label{1.3} -\int_{\Omega}\langle\nabla u,\nabla w\rangle e^{\langle\nu,X\rangle_{g_{0}}}dv
=\int_{\Omega}(\mathfrak{L}_{\nu}w)ue^{\langle\nu,X\rangle_{g_{0}}}dv=\int_{\Omega}(\mathfrak{L}_{\nu}u)we^{\langle\nu,X\rangle_{g_{0}}}dv.
\end{aligned}
\end{equation}
From more analytic viewpoint, just like $\mathfrak{L}$ operator and Witten-Laplacian, it is of great importance to explore some analytic properties of Xin-Laplacian. For example, one can consider  Liouville property, spectrum of Xin-Laplacian, mean value inequality, Gauss
maps, heat kernel associated with the  Xin-Laplacian and so on.
In particular, the first eigenvalue will lead to a lot of very profound results in  understanding some geometric structure of translating solitons although we does not cover this aspect in this paper.

Let $\Omega$ be a bounded domain on an $n$-dimensional Riemannian manifold $\mathcal{M}^{n}$ with piecewise smooth boundary $\partial\Omega$. We consider  Dirichlet eigenvalue problem of   Xin-Laplacian on complete Riemannian manifolds as follows:
\begin{equation}\label{diri-prob}
 \mathfrak{L}_{\nu}u +\Lambda u=0, \ \  {\rm in} \ \ \ \ \Omega,  \ \ {\rm and}\ \
  \ u=0, \ \   {\rm on} \ \ \partial \Omega.\end{equation}
Assume that $\Lambda_{k}$ denotes the $k^{th}$ eigenvalue corresponding to the eigenfunction $u_{k}$. Then, the eigenvalue problem {\rm \eqref{diri-prob}} has real and discrete spectrum satisfying the following inequalities:
$
0<\Lambda_{1}<\Lambda_{2}\leq\Lambda_{3}\leq\cdots\leq\Lambda_{k}\leq\cdots\uparrow+\infty,
$
where each eigenvalue is repeated according to its multiplicity.

On one hand, suppose that $M^{n}$ is an $n$-dimensional Euclidean space $\mathbb{R}^{n}$, Payne,
P\'{o}lya and Weinberger \cite{PPW} investigated the eigenvalues for Dirichlet eigenvalue problem \eqref{diri-prob} of Laplacian and obtained a universal inequality as follows:

\begin{equation}\label{ppw-ineq}\Lambda_{k+1}-\Lambda_{k}\leq\frac{4}{nk}\sum^{k}_{i=1}\Lambda_{i}.\end{equation}Here, the words ``universal inequality" means that  the spectrum is subject to ¡°universal bounds¡± by which certain expressions
involving eigenvalues dominate others with no reference to the geometry of bounded domain $\Omega$ but  reference to the dimension $n$. The study of the universal inequalities are stemmed from Payne, P\'{o}lya and Weinberger's important work in 1956 (cf. \cite{PPW}).
Furthermore, in various settings, many mathematicians extended the universal inequality of Payne, P\'{o}lya and Weinberger.  In particular, Hile and Protter \cite{HP} proved the following universal inequality of eigenvalues:

\begin{equation}\label{hp-ineq}\sum^{k}_{i=1}\frac{\Lambda_{i}}{\Lambda_{k+1}-\Lambda_{i}}\geq\frac{nk}{4},\end{equation}
which is sharper than inequality \eqref{ppw-ineq} given by Payne, P\'{o}lya and Weinberger. Furthermore, an amazing contribution to eigenvalue inequality is that Yang \cite{Y} (cf. \cite{CY2})  obtained a very sharp universal inequality:

\begin{equation}\label{y1-ineq}\sum^{k}_{i=1}(\Lambda_{k+1}-\Lambda_{i})^{2}\leq\frac{4}{n}\sum^{k}_{i=1}(\Lambda_{k+1}-\Lambda_{i})\Lambda_{i}.\end{equation}
From \eqref{y1-ineq}, one can obtain

\begin{equation}\label{y2-ineq}\Lambda_{k+1}\leq\frac{1}{k}(1+\frac{4}{n})\sum^{k}_{i=1}\Lambda_{i}.\end{equation}
The inequalities \eqref{y1-ineq} and \eqref{y2-ineq} are called by Ashbaugh Yang's first inequality and second inequality,
respectively (cf. \cite{A1}, \cite{A2}). In fact, Chebyshev's inequality implies following connections
$\eqref{y1-ineq}\Rightarrow \eqref{y2-ineq} \Rightarrow \eqref{hp-ineq} \Rightarrow \eqref{ppw-ineq}.$

Let $\Psi$ denote the set of all isometric immersions from $\mathcal{M}^{n}$ into the Euclidean space $\mathbb{R}^{n+p}$. In an important literature \cite{CC}, Chen and Cheng investigated Dirichlet problem of Laplacian on the Riemannian manifolds in 2008. In details, based on an extrinsic method on the mean curvature of the immersion, they  proved
\begin{equation}
\begin{aligned}
\label{c-cheng-1} \sum^{k}_{i=1}(\Lambda_{k+1}-\Lambda_{i})^{2}
\leq\frac4n\sum^{k}_{i=1}(\Lambda_{k+1}-\Lambda_{i})\left(\Lambda_i+\frac{1}{4}C_{1}\right),
\end{aligned}
\end{equation}
where $C_{1}=\inf_{\psi\in\Psi}\max_{\Omega}n^{2}H^{2}.$

On the other hand, letting $\Omega$ be a bounded domain  on the plane $\mathbb{R}^{2}$, Payne, P\'{o}lya and Weinberger \cite{PPW} proved that its
lower order eigenvalues satisfy

\begin{equation}\label{ppw-2}\Lambda_{2}+\Lambda_{3}\leq6\Lambda_{1},\end{equation}
which leads to a famous conjecture for
$\Omega\subset\mathbb{R}^{n}$ as follows (see \cite{Ash}):

\begin{GPPW}Let $\Omega$ be a bounded domain on an $n$-dimensional Euclidean space $\mathbb{R}^{n}$. Assume that $\Lambda_{i}$ is the $i$-th eigenvalue of Dirichlet problem
\eqref{diri-prob} for the Laplace operator on $\mathbb{R}^{n}$. Then,    inequality
\begin{equation}
\frac{\Lambda_{2} +\Lambda_{3} +\cdots+ \Lambda_{n+1}}{
\Lambda_{1}}\leq n\frac{\Lambda_{2}(\mathbb{B}^{n})}{\Lambda_{1}(\mathbb{B}^{n})}\end{equation}holds, where $\Lambda_{i}(\mathbb{B}^{n})(i=1,2)$ denotes the $i^{th}$ eigenvalue of Laplacian on the ball $\mathbb{B}^{n}$ with the same volume as the bounded domain $\Omega$, i.e., $Vol(\Omega)=Vol(\Omega^{\ast})$.\end{GPPW} Attacking this conjecture, Brands \cite{Bran} improved
Payne, P\'{o}lya and Weinberger's inequality \eqref{ppw-2} to the following: $\Lambda_{2}+\Lambda_{3}  \leq\Lambda_{1}(3 + \sqrt{7}),
$ when $n=2$.
Furthermore, Hile and Protter \cite{HP} obtained
$\Lambda_{2} +\Lambda_{3}
\leq 5.622\Lambda_{1}.$ In 1980, Marcellini \cite{Mar} proved $\Lambda_{2}+\Lambda_{3} \leq(15 + \sqrt{345})/6\Lambda_{1}.$ In 2011, by a new  approach, Chen and Zheng \cite{CZh} proved $\Lambda_{2}+\Lambda_{3} \leq5.3507\Lambda_{1}.$ For general case,  Ashbaugh and Benguria
\cite{AB1} made a fundamental contribution for establishing a surprising universal inequality as follows:

\begin{equation}\label{1.16}
\frac{\Lambda_{2} +\Lambda_{3} +\cdots+ \Lambda_{n+1}}{
\Lambda_{1}}\leq n + 4.\end{equation} for $\Omega\subset\mathbb{R}^{n}$, in 1993. For more references on the solution of
this conjecture, we refer the readers to
\cite{AB2,AB3,Chit,HP} and references therein. In particular, an amazing breakthrough was made by Ashbaugh and Benguria in \cite{AB2}(or see\cite{AB3}). They  affirmatively settled the general Payne, P\'{o}lya and Weinberger's Conjecture under certain special case. More specifically, by dealing with some good properties of Bessel
functions, Ashbaugh and Benguria  proved a famous conjecture  listed in
problem collection of Yau \cite{Yau}(or cf. \cite{Ash}) as follows:
\begin{PPW}Let $\Omega$ be a bounded domain on an $n$-dimensional Euclidean space $\mathbb{R}^{n}$. Assume that $\Lambda_{i}$ is the $i$-th eigenvalue of the Dirichlet problem
\eqref{diri-prob} for the Laplace operator on $\mathbb{R}^{n}$. Then,  the following eigenvalue inequality

\begin{equation*}
\frac{\Lambda_{2}}{
\Lambda_{1}}\leq \frac{\Lambda_{2}(\mathbb{B}^{n})}{\Lambda_{1}(\mathbb{B}^{n})}\end{equation*}holds, where $\Lambda_{i}(\mathbb{B}^{n})(i=1,2)$ denotes the $i^{th}$ eigenvalue of Laplacian on the ball $\mathbb{B}^{n}$ with the same volume as the bounded domain $\Omega$, i.e., $Vol(\Omega)=Vol(\Omega^{\ast})$.\end{PPW} In 2008, Chen and Cheng
\cite{CC} proved \eqref{1.16} still holds when $\Omega$ is a
bounded domain in a complete Riemannian manifold isometrically
minimally immersed in $\mathbb{R}^{n+p}$ .
Furthermore, Ashbaugh and Benguria \cite{AB1} (cf. Hile and Protter \cite{HP} ) improved the above result
to the following interesting universal inequality:

\begin{equation}\frac{\Lambda_{2} + \Lambda_{3}+ \cdots+ \Lambda_{n+1}}{\Lambda_{1}}\leq n + 3 +\frac{\Lambda_{1}}{\Lambda_{2}}.
\end{equation}
Very recently, Cheng and Qi \cite{CQ} have proved that, for any $1\leq j\leq n + 2$, eigenvalues satisfy at
least one of the following:

\begin{equation*}\begin{aligned} {\rm(1)}\ \
\frac{\Lambda_{2}}{\Lambda_{1}}
< 2-
\frac{\Lambda_{1}}{
\Lambda
_{j}}, \ \ \
 {\rm(2)}\ \ \frac{\Lambda_{2} + \Lambda_{3}+ \cdots+ \Lambda_{n+1}}{\Lambda_{1}}\leq n + 3 +\frac{\Lambda_{1}}{\Lambda_{j}}.
\end{aligned}\end{equation*}In 2002, Levitin and Parnovski \cite{LP} proved an algebraic inequality,  and by using this algebraic inequality, they  generalized  \eqref{1.16}  to
\begin{equation}\label{LP1}
\frac{\Lambda_{j+1} +\Lambda_{j+2} +\cdots+ \Lambda_{j+n}}{
\Lambda_{j}}\leq n + 4,
\end{equation}
where $j$ is any positive integer. For general Riemannian manifold $\mathcal{M}^{n}$ isometrically immersed into the $(n+p)$-dimensional Euclidean space $\mathbb{R}^{n+p}$, in \cite{CC}, Chen and Cheng obtained
\begin{equation}
\begin{aligned}
\label{c-cheng-2} \frac{\lambda_{2} +\lambda_{3} +\cdots+ \lambda_{n+1}}{
\lambda_{1}}\leq n + 4
\end{aligned}
\end{equation}
where $\lambda_{i}=\Lambda_{i}+\frac{1}{4}\inf_{\psi\in\Psi}\max_{\Omega}n^{2}H^{2},$ and   $H$ denotes the mean curvature of $\mathcal{M}^{n}$ defined by \eqref{H-mean-H-equa}.
Also, Soufi, Harrell, Ilias and other mathematicians  made many very important contributions to eigenvalue problem of some self-adjoint elliptic differential operators. In particular, Soufi, Harrell, Ilias studied the eigenvalues of Schr\"{o}dinger operator, and  by the same algebraic argument, they established some interesting inequalities of Payne-P\'{o}lya-Weinberger type in \cite{SHI}, which generalizes inequality \eqref{c-cheng-2}.

Next, let $\mathcal{M}^{n}$ be an $n$-dimensional compact Riemannian manifolds without boundary. We consider the following closed eigenvalue problem of the differential operator $\mathfrak{L}_{\nu}$ on the Riemannian manifolds $\mathcal{M}^{n}$:
\begin{equation}\label{closed-prob}
\mathfrak{L}_{\nu}\overline{u} +\overline{\Lambda} \overline{u}=0, \ \  {\rm in} \ \ \ \ \mathcal{M}^{n}.\end{equation}
Let $\overline{\Lambda}_{k}$ denote the $k$-th eigenvalue of the closed eigenvalue problem \eqref{closed-prob}, which is corresponding to the eigenfunction $\overline{u}_{k}$. Similarly, the spectrum of the eigenvalue problem \eqref{closed-prob}  is
discrete and satisfies
$
0=\overline{\Lambda}_{0}<\overline{\Lambda}_{1}\leq\overline{\Lambda}_{2}\leq\cdots\leq\overline{\Lambda}_{k}\leq\cdots\rightarrow+\infty,
$
where each eigenvalue is repeated according to its multiplicity. Clearly, when $\nu$ vanishes, closed eigenvalue problem \eqref{closed-prob} becomes a classical closed eigenvalue problem of Beltrami-Laplacian:
\begin{equation}\label{closed-prob-2}
\Delta \overline{u} +\overline{\Lambda} \overline{u}=0, \ \  {\rm in} \ \ \ \ \mathcal{M}^{n}.\end{equation}
Assume that $\overline{\Lambda}_{k}$ denotes the $k^{th}$ eigenvalue corresponding to the eigenfunction $\overline{u}_{k}$.
Let $\overline{\Gamma}_{i}$ be the $i$-th distinct eigenvalue of the closed eigenvalue problem \eqref{closed-prob-2} of Beltrami Laplacian on a compact Riemannian manifold without boundary, where $i=0,1,2,\cdots$. In other words, without counting multiplicity of each eigenvalue,   one has the following strict inequalities: $0=\overline{\Gamma}_{0}<\overline{\Gamma}_{1}<\overline{\Gamma}_{2}<\overline{\Gamma}_{3}<\cdots\uparrow+\infty.$

\vskip3mm

\noindent\textbf{Motivation.} It is a very fundamental problem to investigate the eigenvalues of some elliptic operators on the Riemannian manifolds.
Usually, there are two important problems to be considered in spectral geometry: From an analytic perspective, given some geometric and topological structures on the manifolds, one tents to determine the dates or demonstrate certain behaviors of spectrum of elliptic operators; Conversely, from a geometric viewpoint, one always wants to obtain some information on the topology and geometry of manifolds when some spectrum dates are given. The motivation of this paper focuses on the former part. Although we do  does not address the latter part, a conjecture is proposed deriving from some estimates for upper bounds of the closed eigenvalue problem \eqref{closed-prob-2} of  Beltrami-Laplacian.
As the authors know, there is few of investigation for the spectrum of Xin-Laplacian. Thus, it is very urgent for us to consider the eigenvalue problem of Xin-Laplacian.
Inspired by the previous work and the above statements, it is natural for us to discuss the following problem.

\begin{prob-a}Can we establish some inequalities for the lower and higher order eigenvalues of Dirichlet problem \eqref{diri-prob}? Furthermore, for translating solitons, whether the spectral behavior of Xin-Laplacian has a similar rigidity just like the Dirichlet Laplacian on the domain of the Euclidean space or not? \end{prob-a}

In 1982, Yau posed a famous conjecture (see Conjecture \ref{yau-conj} in Section \ref{sec6}) and
his conjecture attracted the attention of many mathematicians, and we will briefly describe its progress in Subsection \ref{subsec6.3}. Up to now, it was far from settled. Yau's conjecture is concerned with the first eigenvalues, which is solved by Tang and Yan in \cite{TY1} when the hypersurface is assumed to be isoparametric. The authors think that the second eigenvalue is also important subject worthy of consideration, and  thus it is natural to ask the following question.

\begin{prob-b}Assume that $\mathcal{M}^{n}$ is an $n$-dimensional minimal hypersurface embedded into the unit sphere $\mathbb{S}^{n+1}(1)$. How can we estimate accurately the second nontrivial eigenvalue (without counting the multiplicities of eigenvalues) of eigenvalue problem \eqref{closed-prob-2}? Furthermore, suppose that $\mathcal{M}^{n}$ is a isoparametric hypersurface or focal submanifold of the unit sphere, can we directly calculate the date of the second nontrivial eigenvalue? \end{prob-b}

Supposing that $\mathcal{M}^{n}$ is an $n$-dimensional isoparametric hypersurface embedded into an $(n+1)$-dimensional unit sphere,  who is not isometric to a unit sphere,  some outstanding  literatures  indicated that $2n$ is an eigenvalue of Beltrami-Laplacian, see \cite{Sol3,TY3}. However, the available results still do not adequately solve the following problem.

\begin{prob-c}
Assume that $M^{n}$ is an $n$-dimensional isoparametric hypersurface embedded into an $(n+1)$-dimensional unit sphere and $\mathcal{M}^{n}$  is not isometric to a unit sphere. Whether   $2n$ is the second non-zero eigenvalue for the closed eigenvalue problem \eqref{closed-prob-2} of Beltrami-Laplacian or not? \end{prob-c}

For some special examples, for example in \cite{Mut,Sol1,Sol2,TXY} and the references therein, Problem C is partially well solved. However, for lots of the other cases, it is still unknown.
As we know, for Yau's conjecture, it is extremely difficult to prove that the coordinate function is the first eigenfunction. Likewise, it is even harder to prove that the eigenfunction corresponding to $2n$ is exactly the second eigenfunction  except for some special cases.

In this paper, we make an affirmative answer to Problem A and partially answer Problem B. Moreover, one of contributions of this paper to Problem C is that our result further hints that $2n$ may be the second eigenvalue of Beltrami-Laplacian on the isoparametric hypersurfaces. Based on those arguments, we propose some conjectures.
This paper is organized as follows.

In Section \ref{sec2}, we prove several auxiliary lemmas. Applying those auxiliary lemmas, we establish some general formulas for Dirichlet eigenvalue problem \eqref{diri-prob}.

Furthermore,  applying those general formulas, we prove the following eigenvalue inequalities in Section \ref{sec3}:
\begin{equation}
\begin{aligned}
\sum^{k}_{i=1}(\Lambda_{k+1}-\Lambda_{i})^{2}
\leq\frac4n\sum^{k}_{i=1}(\Lambda_{k+1}-\Lambda_{i}) \left(\Lambda_{i}+D_{1}\Lambda_{i}^{\frac{1}{2}}+\frac{1}{4}D_{1}^{2}+\frac{1}{4}C_{1}\right),
\end{aligned}
\end{equation}
and

\begin{equation}
\begin{aligned}
\sum_{l=1}^{n}\Lambda_{j+l}\leq
(4+n)\Lambda_{i}+4D_{1}\Lambda_{i}^{\frac{1}{2}}+D_{1}^{2}+C_{1},
\end{aligned}
\end{equation} where $C_{1}=\inf_{\psi\in \Psi}\max_{\Omega}n^{2}H^{2}\ \ {\rm and}  \ \ D_{1}= \max_{\Omega} |\nu^{\top}|_{g_{0}}.$ See Theorem \ref{thm1.1} and Theorem \ref{thm1.2} for details.  Observing the right hand of \eqref{L-equ}, we know that the Xin-Laplacian not only depends on the metric $g$ on the Riemannian manifold $\mathcal{M}^{n}$ but also depends on the standard metric $g_{0}$ on the Euclidean space. Therefore, it is different from the Witten-Laplacian, which only depends on  the Riemannian metric on $\mathcal{M}^{n}$.
It is well know that the Witten-Laplacian is unitarily equivalent to the Schr\"{o}dinger operator, which means that one can  estimate the eigenvalues of Witten-Laplacian by applying Schr\"{o}dinger operator to Witten Laplacian. See {\rm \cite{Set}} for details. However, Xin-Laplacian is not  unitarily equivalent to the Schr\"{o}dinger operator. As a consequence, some methods associated with  unitarily equivalent no longer works in our situations.
Therefore, we remark that our method is different from the method due to Levitin and Parnovski \cite{LP}, where they utilized some algebraic techniques to prove some desired results.

In  Section \ref{sec4}, we discuss the eigenvalues of $\mathfrak{L}_{II}$ operator on the translating solitons. To be special, we obtain the following universal inequalities:
  \begin{equation}
\begin{aligned}\label{1.17}
\sum^{k}_{i=1}(\Lambda_{k+1}-\Lambda_{i})^{2}
\leq\frac{4}{n}\sum^{k}_{i=1}(\Lambda_{k+1}-\Lambda_{i})
\left(\Lambda_{i}+\Lambda_{i}^{\frac{1}{2}}+\frac{n^{2}}{4} \right),
\end{aligned}
\end{equation}and
\begin{equation}
\begin{aligned}
  \sum^{n}_{k=1}\Lambda_{j+k} \leq(n+4)\Lambda_{j} +n^{2}+4\Lambda_{i}^{\frac{1}{2}}.
\end{aligned}
\end{equation}
See  Theorem \ref{thm5.1} and Theorem \ref{thm5.2}. One could hope that eigenvalue inequalities are universal for the Dirichlet problem of some elliptic operators on Riemannian manifolds, but, unfortunately, this is not always possible.  In general, it is not easy to obtain universal inequalities for weighted Laplacian and even Laplacian on the complete Riemannian manifolds. Therefore, our work can be regarded as a new contribution to universal inequality. Furthermore, by using \eqref{1.17}, we give some estimates for the upper bounds of the $k$-th eigenvalue and gaps of the consecutive eigenvalues of  $\mathfrak{L}_{II}$ operator on the translating solitons.

As some  further applications, we discuss the eigenvalues on the minimal submanifolds on the Euclidean spaces, submanifolds on the unit spheres, projective spaces in Section \ref{sec5}. In  addition, we also consider the eigenvalues  on some manifolds admitting with special functions such as Cartan-Hadamard manifolds, product manifolds and homogeneous manifolds and so on in this section. We refer the readers to Corollary \ref{corr-6.1}-\ref{corr-6.5} for details.

Before starting Problem B, we motivate the  study of closed eigenvalue problem \eqref{closed-prob} and establish some eigenvalue inequalities in Section \ref{sec6}. Furthermore, as some remarkable applications, we prove some eigenvalue inequalities of Xin-Laplacian on the  minimal submanifolds in the unit sphere and generalize
the Reilly's result on the first  eigenvalue of the Beltrami-Laplacian.
More importantly, we suppose that $\mathcal{M}^{n}$ is an $n$-dimensional compact minimal isoparametric hypersurface in the unit sphere $\mathbb{S} ^{n+1}(1)$, and prove that eigenvalues of the closed eigenvalue problem \eqref{closed-prob-2} of the Beltrami-Laplacian satisfy
\begin{equation}\label{iso-ineq}
\frac{1}{n} \sum_{k=1}^{n} \overline{\Lambda}_{n_{0}+k} \leq 2 n+4,
\end{equation}where $n_{0}$ denotes the value of the multiplicity of the first eigenvalue, which gets a very sharp estimation for the upper bound of the second eigenvalues as follows: $$\overline{\Gamma}_{2}\leq2n+4.$$
Here, we do not  count the multiplicity of eigenvalues.  As a byproduct, our result further hints that $2n$ could be the second non-zero eigenvalue in term of the isoparametric hypersurfaces of OT-FKM type. For further details, we refer the readers to Remark \ref{Rem-Sol3}. Clearly, focal submanifolds are some important minimal submanifolds of the unit spheres.
In the remainder part of this section, we also discuss the eigenvalues of the Laplacian on them and some upper bounds are obtained.

Based on some arguments for the eigenvalues of Xin-Laplacian on the complete Riemannian manifolds, several conjectures are posed in  Section \ref{sec7}. In particular, Conjecture {\rm \ref{Z-conj}} is closely related to  Yau's conjecture \ref{yau-conj}. Here, it is necessary for us to emphasize that conjecture  \ref{Z-conj} is presented based  entirely on a very sharp estimates for the upper bound of  the second non-zero eigenvalue (without counting the multiplicity of eigenvalues) in Section \ref{sec6}, and many important examples hint that this conjecture is true. In addition, the second  eigenvalue will perfectly characterize the isoparametric hypersurfaces if it is true.  Furthermore, assumed that  minimal hypersurfaces has constant scalar curvature(without isoparametric assumption), the first author present a rigidity conjecture (see Conjecture \ref{regidity-conj}), which is related to Yau's conjecture \ref{yau-conj} and  Chern's conjecture \ref{chern-conj}.  To solve   conjectures \ref{Z-conj} and \ref{regidity-conj}, it seems to be very crucial for us to have an in-depth understanding for the topology of minimal hypersurfaces on the unit spheres. Therefore, we think that the study of those conjectures maybe have a far-reaching impact on the topology of minimal hypersurfaces in the unit sphere, especially for the isoparametric theory.

\section{Several Auxiliary Lemmas and General Formulas} \label{sec2}
\vskip3mm
\subsection{General Formulas}\label{subsec2.1}
In this subsection, we would like to establish two general formulas, which will play critical roles in the proofs of main results.  Our first general formula says the following.

\begin{prop}
\label{prop2.2} Let $\phi_l$, $l=1, 2, \cdots, m$,  be  smooth
functions on an $n$-dimensional complete  Riemannian manifold $\mathcal{M}^{n}$
and $\Lambda_{k}$ the $k^{\text{th}}$ eigenvalue of \eqref{diri-prob}. Then, for any $j=1, 2, \cdots$,  there exists an
orthogonal matrix $A=(a_{ls})_{m\times m}$ such that $\Phi_l=\sum_{s=1}^ma_{ls}\phi_s$
satisfy
\begin{equation}\label{2.7}
\sum^{m}_{l=1}(\Lambda_{j+l}-\Lambda_{j})\|u_j\nabla
\Phi_{l}\|^{2}_{\Omega} \leq
\sum^{m}_{l=1}\int_{\Omega}\big(u_{j}\mathfrak{L}_{\nu}\Phi_{l} +2\langle\nabla
\Phi_{l},\nabla u_{j}\rangle_{g}\big)^{2}e^{\langle\nu,X\rangle_{g_{0}}}dv,
\end{equation}
where $u_{j}$ is an orthonormal eigenfunction corresponding to
eigenvalue $\Lambda_{j}$ and $$\|f(x)\|_{\Omega}^{2}=\int_{\Omega}f(x)^{2}e^{\langle\nu,X\rangle_{g_{0}}}dv.$$
\end{prop}
In order to prove Proposition \ref{prop2.2}, we need the following auxiliary lemmas.

\begin{lem}
\label{lem2.1} Let $\phi$ be a smooth function on an $n$-dimensional
complete  Riemannian manifold $\Omega$. Assume that $\Lambda_{i}$  is  the $i^{\text{th}}$ eigenvalue of the
Dirichlet eigenvalue problem \eqref{diri-prob} and $u_{i}$ is an
orthonormal eigenfunction corresponding to $\Lambda_{i}$ such that $\mathfrak{L}_{\nu}u_{i} =-\Lambda_{i}u_{i},$ and $ \int_{\Omega}
u_{i}u_{j}e^{\langle\nu,X\rangle_{g_{0}}}dv=\delta_{ij},
$ where $i,j=1,2,\cdots$.   Then, for any $j=
1, 2, \cdots,$  the following equation
\begin{equation}\label{3.1}
\|u_{j}\nabla \phi\|_{\Omega}^{2}
=\sum^{\infty}_{k=1}(\Lambda_{k}-\Lambda_{j})\sigma_{jk}^{2},
\end{equation}holds,
where $\sigma_{jk}=\int_{\Omega}\phi u_{j}u_{k}e^{\langle\nu,X\rangle_{g_{0}}}dv.$
\end{lem}

\begin{proof}
Since $\{u_{k}\}^{\infty}_{k=1}$ is an orthonormal basis of the
weighted $L^{2}(\Omega)$, then, for any $j$, where $j=1,2,\cdots$, we know that
\begin{equation}\label{2.2}
\phi u_{j}=\sum^{\infty}_{k=1}\sigma_{jk}u_{k}.
\end{equation}
According to Parseval's identity,  it is not difficult to check that
\begin{equation*}
\|\phi u_{j}\|^{2}_{\Omega}=\sum^{\infty}_{k=1}\sigma_{jk}^{2}.
\end{equation*}
By a simple computation, we immediately derive
\begin{equation*}
\begin{aligned}
\int_{\Omega}(\mathfrak{L}_{\nu}\left(\phi u_{j})-\phi\mathfrak{L}_{\nu}u_{j}\right)u_{k}e^{\langle\nu,X\rangle_{g_{0}}}dv
=(\Lambda_{j}-\Lambda_{k})\int_{\Omega}\phi u_{j}u_{k}e^{\langle\nu,X\rangle_{g_{0}}}dv,
\end{aligned}
\end{equation*}
and
\begin{equation*}
\mathfrak{L}_{\nu}(\phi u_{j}) =\phi\mathfrak{L}_{\nu}u_{j}+u_{j}\mathfrak{L}_{\nu}\phi+2\langle\nabla
\phi,\nabla u_{j}\rangle_{g}.
\end{equation*}
Therefore, we have
\begin{equation}
\label{2.3}\int_{\Omega}(u_{j}\mathfrak{L}_{\nu}\phi+2\langle\nabla \phi,\nabla
u_{j}\rangle_{g})u_{k}e^{\langle\nu,X\rangle_{g_{0}}}dv
=(\Lambda_{j}-\Lambda_{k})\int_{\Omega}\phi u_{j}u_{k}e^{\langle\nu,X\rangle_{g_{0}}}dv.
\end{equation}
Furthermore, from \eqref{2.2}, we deduce
\begin{equation}
\begin{aligned}
\label{2.4}
\sum^{\infty}_{k=1}(\Lambda_{k}-\Lambda_{j})\sigma_{jk}^{2}
&=\sum^{\infty}_{k=1}(\Lambda_{k}-\Lambda_{j})\left(\int_{\Omega}\phi u_{j}u_{k}e^{\langle\nu,X\rangle_{g_{0}}}dv\right)^{2}
\\&=\sum^{\infty}_{k=1}\Lambda_{k}\left(\int_{\Omega}\phi u_{j}u_{k}e^{\langle\nu,X\rangle_{g_{0}}}dv\right)^{2}-\Lambda_{j}\|\phi u_{j}\|_{\Omega}^{2}.
\end{aligned}\end{equation}
By using \eqref{2.2}, we derive
\begin{equation*}
\mathfrak{L}_{\nu}(\phi u_{j})
=\sum^{\infty}_{k=1}\sigma_{jk}\mathfrak{L}_{\nu}u_{k}=-\sum^{\infty}_{k=1}\sigma_{jk}\Lambda_{k}u_{k},
\end{equation*}
which implies
\begin{equation}
\label{2.5}\phi u_{j}\mathfrak{L}_{\nu}(\phi u_{j})
=-\sum^{\infty}_{k=1}\sigma_{jk}\Lambda_{k}u_{k}\phi u_{j}.
\end{equation}
Hence, it follows from \eqref{2.5} that
\begin{equation}
\begin{aligned}
\label{2.6} \int_{\Omega}\phi u_{j}\mathfrak{L}_{\nu}(\phi u_{j})e^{\langle\nu,X\rangle_{g_{0}}}dv
&=-\sum^{\infty}_{k=1}\int_{\Omega}\sigma_{jk}\Lambda_{k}u_{k}\phi u_{j}e^{\langle\nu,X\rangle_{g_{0}}}dv
\\&=-\sum^{\infty}_{k=1}\Lambda_{k}\int_{\Omega}\phi u_{j}u_{k}e^{\langle\nu,X\rangle_{g_{0}}}dv\int_{\Omega}u_{k}\phi u_{j}e^{\langle\nu,X\rangle_{g_{0}}}dv
\\&=-\sum^{\infty}_{k=1}\Lambda_{k}\left(\int_{\Omega}\phi u_{j}u_{k}e^{\langle\nu,X\rangle_{g_{0}}}dv\right)^{2}.
\end{aligned}
\end{equation}
Combining   \eqref{2.4} with \eqref{2.6}, we have
\begin{equation*}
\begin{aligned}
\sum^{\infty}_{k=1}(\Lambda_{k}-\Lambda_{j})\sigma_{jk}^{2}
&=\int_{\Omega}(-\phi u_{j}\mathfrak{L}_{\nu}(\phi u_{j})-\Lambda_{j}\phi^{2}u_{j}^{2})e^{\langle\nu,X\rangle_{g_{0}}}dv
\\&=\int_{\Omega}(-\phi u_{j}^{2}\mathfrak{L}_{\nu}\phi-2\langle\nabla\phi,\nabla u_{j}\rangle_{g} \phi u_{j})e^{\langle\nu,X\rangle_{g_{0}}}dv
\\&=\int_{\Omega}\left(|\nabla\phi|_{g}^{2}u_{j}^{2}+\frac{1}{2}\langle\nabla \phi^{2},\nabla u_{j}^{2}\rangle_{g}
-\frac{1}{2}\langle\nabla\phi^{2},\nabla
u_{j}^{2}\rangle_{g}\right)e^{\langle\nu,X\rangle_{g_{0}}}dv
\\&=\|u_{j}\nabla \phi\|^{2}_{\Omega}.
\end{aligned}
\end{equation*}
This completes the proof of Lemma \ref{lem2.1}.
\end{proof}
By making use of the Lemma \ref{lem2.1}, we give the proof of Proposition \ref{prop2.2}.

\vskip 3mm
\noindent \emph{Proof of Proposition }\ref{prop2.2}.  For any $j= 1, 2, \cdots$,
we consider the following $m\times m$-matrix:
\begin{equation*}
C:=\left(\int_{\Omega}\big(u_{j}\mathfrak{L}_{\nu}\phi_{l} +2\langle\nabla
\phi_{l},\nabla u_{j}\rangle_{g}\big)u_{j+s}e^{\langle\nu,X\rangle_{g_{0}}}dv\right)_{m\times m}.
\end{equation*}
According to the Gram-Schmidt orthogonalization, we know that there exists an orthogonal
matrix $A=(a_{ls})$ such that
\begin{equation*}
Q=AC=(q_{ls})_{m\times m}
=
\begin{aligned}\left(
                 \begin{array}{cccc}
                   q_{11} &q_{12} & \cdots& q_{1m} \\
                   0 & q_{22} &\cdots & q_{2m} \\
                   \vdots & \vdots & \ddots & \vdots \\
                   0 & 0 &\cdots & q_{mm} \\
                 \end{array}
               \right),
\end{aligned}
\end{equation*} that is,
\begin{equation*}
\begin{aligned}
q_{ls}&=\sum_{i=1}^{m}a_{li}\int_{\Omega}\biggl(u_{j}\mathfrak{L}_{\nu}\phi_{i}+2\langle\nabla
\phi_{i},\nabla u_{j}\rangle_{g}\biggl)u_{j+s}e^{\langle\nu,X\rangle_{g_{0}}}dv\\
&=\displaystyle\int_{\Omega}\left(u_{j}\mathfrak{L}_{\nu}\left(\sum_{i=1}^{m}a_{li}\phi_{i}\right)
+2\langle\nabla\left(\sum_{i=1}^{m}a_{li}\phi_{i}\right),\nabla
u_{j}\rangle_{g}\right)u_{j+s}e^{\langle\nu,X\rangle_{g_{0}}}dv,
\end{aligned}
\end{equation*}
with $q_{ls}=0$ for $l>s$. For
$
\Phi_{l}=\sum_{i=1}^{m}a_{li}\phi_{i},
$
we have
\begin{equation*}
\begin{aligned}
q_{ls}&=\displaystyle\int_{\Omega}\left(u_{j}\mathfrak{L}_{\nu}\Phi_{l}+2\left\langle\nabla
\Phi_{l},\nabla u_{j}\right\rangle_{g}\right)u_{j+s}e^{\langle\nu,X\rangle_{g_{0}}}dv=0, \ \text{for
$l>s$}.
\end{aligned}
\end{equation*}
Applying the Lemma \ref{lem2.1} to functions $\Phi_l$,
we yield
\begin{equation}
\begin{aligned}\label{2.8}
&\|u_j\nabla \Phi_{l}\|^{2}_{\Omega}=\sum^{j-1}_{k=1}(\Lambda_{k}-\Lambda_{j})\beta^{2}_{ljk}
+\sum^{j+l-1}_{k=j}(\Lambda_{k}-\Lambda_{j})\beta^{2}_{ljk}
+\sum^{\infty}_{k=j+l}(\Lambda_{k}-\Lambda
_{j})\beta^{2}_{ljk},
\end{aligned}
\end{equation}where
$
\beta_{ljk}:=\int_{\Omega}\Phi_{l}u_{j}u_{k}e^{\langle\nu,X\rangle_{g_{0}}}dv.
$
According to \eqref{2.3} in place of $\phi$ with  $\Phi_l$, it is easy to verify that
\begin{equation}\label{2.9}
\int_{\Omega}(u_{j}\mathfrak{L}_{\nu}\Phi_l+2\langle\nabla \Phi_l,\nabla
u_{j}\rangle_{g})u_{k}e^{\langle\nu,X\rangle_{g_{0}}}dv
=(\Lambda_{j}-\Lambda_{k})\int_{\Omega}\Phi_lu_{j}u_{k}e^{\langle\nu,X\rangle_{g_{0}}}dv,
\end{equation}
which implies
$$
(\Lambda_{k}-\Lambda_{j})\int_{\Omega}\Phi_lu_{j}u_{k}e^{\langle\nu,X\rangle_{g_{0}}}dv=(\Lambda_{k}-\Lambda_{j})\beta_{ljk}=0, \
\text{for $k=j, j+1, \cdots, j+l-1$}.
$$
From  \eqref{2.8}, we conclude
\begin{equation}
\begin{aligned}\label{2.10}
\|u_j\nabla \Phi_{l}\|^{2}_{\Omega}&\leq
\sum^{\infty}_{k=j+l}(\Lambda_{k}-\Lambda_{j})\beta^{2}_{ljk}.
\end{aligned}
\end{equation}
Hence, from \eqref{2.9}, \eqref{2.10} and Parseval's
identity, we infer that,
\begin{equation*}
\begin{aligned}
\sum^{m}_{l=1}(\Lambda_{j+l}-\Lambda_{j})\|u_j\nabla
\Phi_{l}\|^{2}_{\Omega}&
\leq \sum^{m}_{l=1}\sum^{\infty}_{k=j+l}(\Lambda_{k}-\Lambda_{j})^2\beta^{2}_{ljk}\\
&\leq \sum^{m}_{l=1}\int_{\Omega}\big(u_{j}\mathfrak{L}_{\nu}\Phi_{l}
+2\langle\nabla \Phi_{l},\nabla u_{j}\rangle_{g}\big)^{2}e^{\langle\nu,X\rangle_{g_{0}}}dv.
\end{aligned}
\end{equation*}
The proof ends.
 $$\eqno \Box$$

\noindent In what follows, we would like to prove the second  general formula for eigenvalues, which generalizes
a formula established by  Cheng and Yang in \cite{CY2} for the eigenvalue
problem of the Laplacian. We remark that the original method of this proof  is due to Cheng and
Yang in \cite{CY2}. However, for the convenience of readers, we shall give a
self contained proof.

\begin{prop}\label{prop2.3}
Let $(\mathcal{M}^{n},g)$ be an $n$-dimensional complete noncompact Riemannian manifold.
Assume that $\Lambda_{i}$  is  the $i^{\text{th}}$ eigenvalue of the
Dirichlet eigenvalue problem \eqref{diri-prob} and $u_{i}$ is an
orthonormal eigenfunction corresponding to $\Lambda_{i}$ such that
$\mathfrak{L}_{\nu}u_{i} =-\Lambda_{i}u_{i},$ and
$ \int_{\Omega}
u_{i}u_{j}e^{\langle\nu,X\rangle_{g_{0}}}dv=\delta_{ij},
$ where $i,j=1,2,\cdots$.
Then, for any function $\varphi(x)\in C^{2}(\Omega)$ and any positive
integer $k$,  eigenvalues of the Dirichlet eigenvalue problem \eqref{diri-prob} satisfy
\begin{equation}
\begin{aligned}
\label{2.11}
\sum^{k}_{i=1}(\Lambda_{k+1}-\Lambda_{i})^{2}\|u_{i}\nabla
\varphi\|_{\Omega}^{2} \leq\sum^{k}_{i=1}(\Lambda_{k+1}-\Lambda_{i})
\|2\langle\nabla \varphi,\nabla u_{i}\rangle_{g}+u_{i}\mathfrak{L}_{\nu}\varphi\|_{\Omega}^{2}.
\end{aligned}
\end{equation}

\end{prop}
Before giving the proof of Proposition \ref{prop2.3}, we shall introduce several notations. Set

\begin{equation*}
\sigma_{ij}:=\int_{\Omega}\varphi u_{i}u_{j}e^{\langle\nu,X\rangle_{g_{0}}}dv,
\end{equation*}and

\begin{equation*}
\begin{aligned}
\Theta_{i}:=-\int_{\Omega}\zeta_{i}(u_{i}\mathfrak{L}_{\nu}\varphi+2\langle\nabla
\varphi,\nabla u_{i}\rangle_{g})e^{\langle\nu,X\rangle_{g_{0}}}dv .
\end{aligned}
\end{equation*}where $\varphi(x)\in C^{2}(\Omega)$, and
\begin{equation*}
\zeta_{i}:=\varphi u_{i}-\sum^{k}_{j=1}\sigma_{ij}u_{j}.
\end{equation*}
Define
$$
\tau_{ij}:=-\int_{\Omega}(u_{j}\mathfrak{L}_{\nu}\varphi+2\langle\nabla \varphi,\nabla
u_{j}\rangle_{g}) u_{i}e^{\langle\nu,X\rangle_{g_{0}}}dv.
$$

In order to prove Proposition \ref{prop2.3}, we need the following several auxiliary lemmas. The first auxiliary lemma is show that
$\tau_{ji}$ is skew-symmetric.
\begin{lem}\label{l2.2}Under the assumption of Proposition {\rm\ref{prop2.3}}, we have
\begin{equation}
\label{2.15} \tau_{ij}=(\Lambda_{i}-\Lambda_{j})\sigma_{ij},
\end{equation}and \begin{equation}
\tau_{ij}=-\tau_{ji}.
\end{equation}\end{lem}

\begin{proof}By utilizing \eqref{1.3}, it is easy to verify that
$
\tau_{ij}=(\Lambda_{i}-\Lambda_{j})\sigma_{ij}.
$ By the definition of $\sigma_{ij}$,
we know that $\sigma_{ij}=\sigma_{ji}.$
Therefore, we have
$
\tau_{ij}=-\tau_{ji}.
$
This finishes the proof of this Lemma.

 \end{proof}

Next, we shall give an estimate for the lower bound of $\Theta_{i}$.
\begin{lem}\label{l2.3}Under the assumption of Proposition {\rm\ref{prop2.3}}, we have

\begin{equation}
\begin{aligned}
\label{2.14}  (\Lambda_{k+1}-\Lambda_{i})\|\zeta_{i}\|_{\Omega}^{2}
\leq  \Theta_{i}.
\end{aligned}
\end{equation}\end{lem}

\begin{proof} Since  $u_{j}$ is  an
orthonormal eigenfunction corresponding to the eigenvalue
$\Lambda_j$,
 $\{u_{j}\}^{\infty}_{j=1}$ forms an orthonormal basis of
the weighted $L^{2}(\Omega)$. Furthermore, by the Rayleigh-Ritz inequality, we
have
\begin{equation}
\begin{aligned}
\label{2.12} \Lambda_{k+1}\leq-\frac{\displaystyle
\int_{\Omega}\varphi\mathfrak{L}_{\nu}\varphi e^{\langle\nu,X\rangle_{g_{0}}}dv}{\displaystyle
\int_{\Omega}\varphi^{2}e^{\langle\nu,X\rangle_{g_{0}}}dv},
\end{aligned}
\end{equation}
for any function $\varphi$ satisfing $$\int_{\Omega}\varphi
u_{j}e^{\langle\nu,X\rangle_{g_{0}}}dv=0 ,\ \ 1\leq j\leq k.$$
By a direct calculation, it is not difficult to verify that
\begin{equation}
\begin{aligned}
\label{2.13} \int_{\Omega}\zeta_{i}u_{l}e^{\langle\nu,X\rangle_{g_{0}}}dv&=0,
\end{aligned}
\end{equation}
for $1\leq i, l\leq k.$  Clearly, \eqref{2.12}  implies
\begin{equation*}
\Lambda_{k+1}\leq-\frac{\displaystyle
\int_{\Omega}\zeta_{i}\mathfrak{L}_{\nu}\zeta_{i}e^{\langle\nu,X\rangle_{g_{0}}}dv}{\displaystyle
\int_{\Omega}\zeta^{2}_{i}e^{\langle\nu,X\rangle_{g_{0}}}dv}.
\end{equation*}
Since
\begin{equation*}
\begin{aligned}
\mathfrak{L}_{\nu}\zeta_{i} =u_{i}\mathfrak{L}_{\nu}\varphi-\Lambda_{i}\varphi u_{i}+2\langle\nabla \varphi,\nabla u_{j}\rangle_{g}
+\sum^{k}_{j=1}\Lambda_{j}\sigma_{ij}u_{j},
\end{aligned}
\end{equation*}
from \eqref{2.13}, we have
\begin{equation*}
\begin{aligned}
(\Lambda_{k+1}-\Lambda_{i})\|\zeta_{i}\|_{\Omega}^{2}
\leq-\int_{\Omega}\zeta_{i}(u_{i}\mathfrak{L}_{\nu}\varphi+2\langle\nabla
\varphi,\nabla u_{i}\rangle_{g})e^{\langle\nu,X\rangle_{g_{0}}}dv = \Theta_{i}.
\end{aligned}
\end{equation*} Thus, we finish the proof of this Lemma.

\end{proof}
From Lemma \ref{l2.3}, we have the following lemma.
\begin{lem}\label{l2.4}Under the assumption of Proposition {\rm\ref{prop2.3}}, we have
\begin{equation}
\begin{aligned}
\sum_{i=1}^{k}(\Lambda_{k+1}-\Lambda_{i})^{2}\Theta_{i}
\leq\sum_{i=1}^{k}(\Lambda_{k+1}-\Lambda_{i})\|u_{i}\mathfrak{L}_{\nu}\varphi+2\langle\nabla
\varphi,\nabla u_{i}\rangle_{g}-\sum_{j=1}^{k}\tau_{ij}u_{j}\|_{\Omega}^{2}.
\end{aligned}
\end{equation}
\end{lem}

\begin{proof}
Firstly, we give an estimate for the upper bound of $\Theta_{i}$.
From \eqref{2.14}, \eqref{2.13} and the Cauchy-Schwarz inequality,
we infer
\begin{equation}
\begin{aligned}
\label{2.16}
\Theta_{i}&=-\int_{\Omega}\zeta_{i}\left(u_{i}\mathfrak{L}_{\nu}\varphi+2\langle\nabla
\varphi,\nabla u_{i}\rangle_{g}-\sum_{j=1}^{k}\tau_{ij}u_{j}\right)e^{\langle\nu,X\rangle_{g_{0}}}dv
\\&\leq\left\{\|\zeta_{i}\|_{\Omega}^{2}\|u_{i}\mathfrak{L}_{\nu}\varphi+2\langle\nabla
\varphi,\nabla
u_{i}\rangle_{g}-\sum_{j=1}^{k}\tau_{ij}u_{j}\|_{\Omega}^{2}\right\}^{\frac{1}{2}}.
\end{aligned}
\end{equation}
Uniting \eqref{2.14} and \eqref{2.16}, we obtain
\begin{equation*}
\begin{aligned}
(\Lambda_{k+1}-\Lambda_{i})\Theta^{2}_{i}
&\leq(\Lambda_{k+1}-\Lambda_{i})\|\zeta_{i}\|_{\Omega}^{2}
\|u_{i}\mathfrak{L}_{\nu}\varphi+2\langle\nabla \varphi,\nabla
u_{i}\rangle_{g}-\sum_{j=1}^{k}\tau_{ij}u_{j}\|_{\Omega}^{2}
\\&\leq \Theta_{i}\|u_{i}\mathfrak{L}_{\nu}\varphi+2\langle\nabla \varphi,\nabla
u_{i}\rangle_{g}-\sum_{j=1}^{k}\tau_{ij}u_{j}\|_{\Omega}^{2}.
\end{aligned}
\end{equation*}
Therefore, we have
\begin{equation}
\begin{aligned}
\label{2.17} (\Lambda_{k+1}-\Lambda_{i})^{2}\Theta_{i}
\leq(\Lambda_{k+1}-\Lambda_{i})\|u_{i}\mathfrak{L}_{\nu}\varphi+2\langle\nabla
\varphi,\nabla u_{i}\rangle_{g}-\sum_{j=1}^{k}\tau_{ij}u_{j}\|_{\Omega}^{2}.
\end{aligned}
\end{equation}
Summing on $i$ from $1$ to $k$ for \eqref{2.17}, we derive
\begin{equation*}
\begin{aligned}
\sum_{i=1}^{k}(\Lambda_{k+1}-\Lambda_{i})^{2}\Theta_{i}
\leq\sum_{i=1}^{k}(\Lambda_{k+1}-\Lambda_{i})\|u_{i}\mathfrak{L}_{\nu}\varphi+2\langle\nabla
\varphi,\nabla u_{i}\rangle_{g}-\sum_{j=1}^{k}\tau_{ij}u_{j}\|_{\Omega}^{2},
\end{aligned}
\end{equation*}
as claimed. Thus, the proof of this lemma ends.

\end{proof}

Applying Lemma \ref{l2.2}, Lemma \ref{l2.3} and Lemma \ref{l2.4}, we give the proof of Proposition \ref{prop2.3}.

\vskip 3mm
\noindent \emph{Proof of Proposition} \ref{prop2.3}.
By the definition of $\tau_{ij}$ and \eqref{2.15}, it is not difficult to infer that
\begin{equation}
\begin{aligned}
\label{2.18}& \|u_{i}\mathfrak{L}_{\nu}\varphi +2\langle\nabla \varphi,\nabla
u_{i}\rangle_{g} -\sum^{n}_{j=1}\tau_{ij}u_{j}\|_{\Omega}^{2}
\\&=\|u_{i}\mathfrak{L}_{\nu}\varphi+2\langle\nabla
\varphi,\nabla u_{i}\rangle_{g}\|_{\Omega}^{2}- \sum^{n}_{j=1}\tau_{ij}^{2}
\\&=\|u_{i}\mathfrak{L}_{\nu}\varphi+2\langle\nabla
\varphi,\nabla u_{i}\rangle_{g}\|_{\Omega}^{2}-
\sum^{n}_{j=1}(\Lambda_{i}-\Lambda_{j})^{2}\sigma_{ij}^{2},
\end{aligned}
\end{equation}
According to the definitions of $\Theta_{i}$ and $\zeta_{i}$, it follows
from  \eqref{2.15} that,
\begin{equation}
\begin{aligned}
\label{2.19}\Theta_{i}
&=-\int_{\Omega}\left(\varphi u_{i}-\sum^{k}_{j=0}\sigma_{ij}u_{j}
\right)\Bigg(u_{i}\mathfrak{L}_{\nu}\varphi+2\langle\nabla \varphi,\nabla
u_{i}\rangle_{g}\Bigg)e^{\langle\nu,X\rangle_{g_{0}}}dv
\\&=-\int_{\Omega}
(\varphi u^{2}_{i}\mathfrak{L}_{\nu}\varphi+2\varphi u_{i}\langle\nabla \varphi,\nabla
u_{i}\rangle_{g})e^{\langle\nu,X\rangle_{g_{0}}}dv
\\&\quad~+\sum^{k}_{j=1}a_{ij}\int_{\Omega}u_{j}
(u_{i}\mathfrak{L}_{\nu}\varphi+2\langle\nabla \varphi,\nabla u_{i}\rangle_{g})e^{\langle\nu,X\rangle_{g_{0}}}dv
\\&=-\int_{\Omega}
\left(\varphi\mathfrak{L}_{\nu}\varphi-\frac{1}{2}\mathfrak{L}_{\nu}\varphi^{2}\right)u^{2}_{i}e^{\langle\nu,X\rangle_{g_{0}}}dv+\sum^{k}_{j=1}\sigma_{ij}\tau_{ij}
\\&=\int_{\Omega}
\langle\nabla \varphi,\nabla \varphi\rangle
u^{2}_{i}e^{\langle\nu,X\rangle_{g_{0}}}dv+\sum^{k}_{j=1}(\Lambda_{i}-\Lambda_{j})\sigma_{ij}^{2}.
\end{aligned}
\end{equation}
A simple calculation shows that

\begin{equation}
\begin{aligned}
\label{2.20}
 \sum_{i,j=1}^{k}(\Lambda_{k+1}-\Lambda_{i})^{2}(\Lambda_{i}-\Lambda_{j})\sigma^{2}_{ij}
=-\sum_{i,j=1}^{k}(\Lambda_{k+1}-\Lambda_{i})(\Lambda_{i}-\Lambda_{j})^{2}\sigma^{2}_{ij}.
\end{aligned}
\end{equation}
Furthermore,  uniting \eqref{2.17}, \eqref{2.18}, \eqref{2.19}
and \eqref{2.20}, we get
\begin{equation*}
\begin{aligned}
\sum^{k}_{i=1}(\Lambda_{k+1}-\Lambda_{i})^{2}\|u_{i}\nabla
\varphi\|_{\Omega}^{2} \leq\sum^{k}_{i=1}(\Lambda_{k+1}-\Lambda_{i})
\|2\langle\nabla \varphi,~\nabla u_{i}\rangle_{g}+u_{i}\mathfrak{L}_{\nu}\varphi\|_{\Omega}^{2}.
\end{aligned}
\end{equation*}
Therefore, we finish the proof of this proposition.
$$\eqno\Box$$

\subsection{Extrinsic Formulas}\label{subsec2.2}
\vskip3mm \noindent   Assume that $\left\{e_{1}, \cdots, e_{n}\right\}$ is a local orthonormal basis of $\mathcal{M}^{n}$ with respect to the induced metric $g$, and $\{e_{n+1}, \cdots, e_{n+p}\}$ is the local unit orthonormal normal vector fields. Assume that \begin{equation}\label{H-mean-H-equa}\textbf{H} =\frac{1}{n}\sum_{\alpha=n+1}^{n+p} H^{\alpha} e_{\alpha}=\frac{1}{n}\sum_{\alpha=n+1}^{n+p}\left(\sum_{i=1}^{n} h_{i i}^{\alpha}\right) e_{\alpha},\ \ {\rm and} \ \ H=\frac{1}{n} \sqrt{\sum_{\alpha=n+1}^{n+p}\left(\sum_{i=1}^{n} h_{i i}^{\alpha}\right)^{2}}\end{equation} are the mean curvature vector field and the mean curvature of $\mathcal{M}^{n}$, respectively. In order to prove our main results, we need the
following lemma. A proof of it can be
found in \cite{CC}.

\begin{lem}\label{lem2.5}
For  an $n$-dimensional  submanifold  $\mathcal{M}^{n}$ in Euclidean space
$\mathbb{R}^{n+p}$,   let $x=(x_{1},x_{2},\cdots,x_{n+p})$ is the
position vector of a point $p\in \mathcal{M}^{n}$ with
$x_{\alpha}=x_{\alpha}(y_{1}, \cdots, y_{n})$, $1\leq \alpha\leq
n+p$, where $(y_{1}, \cdots, y_{n})$ denotes a local coordinate
system of $\mathcal{M}^n$. Then, we have
\begin{equation*}
\sum^{n+p}_{\alpha=1}\langle\nabla x_{\alpha},\nabla x_{\alpha}\rangle_{g}= n,
\end{equation*}
\begin{equation*}
\begin{aligned}
\sum^{n+p}_{\alpha=1}\langle\nabla x_{\alpha},\nabla u\rangle_{g}\langle\nabla
x_{\alpha},\nabla w\rangle_{g}=\langle\nabla u,\nabla w\rangle_{g},
\end{aligned}
\end{equation*}
for any functions  $u, w\in C^{1}(\mathcal{M}^{n})$,
\begin{equation*}
\begin{aligned}
\sum^{n+p}_{\alpha=1}(\Delta x_{\alpha})^{2}=n^{2}H^{2},
\end{aligned}
\end{equation*}and
\begin{equation*}
\begin{aligned}
\sum^{n+p}_{\alpha=1}\Delta x_{\alpha}\nabla x_{\alpha}= 0,
\end{aligned}
\end{equation*}
where $H$ is the mean curvature of $\mathcal{M}^{n}$.
\end{lem}

We choose a new coordinate system $\bar{y}=\left(\bar{y}^{1}, \cdots, \bar{y}^{n+p}\right)$ of $\mathbb{R}^{n+p}$ given by
$y-y(P)=\bar{y} A,$ such that $\left(\frac{\partial}{\partial \bar{y}^{1}}\right)_{P}, \cdots,\left(\frac{\partial}{\partial \bar{y}^{n}}\right)_{P} \operatorname{span} T_{P} \mathcal{M}^{n}$, and at $P,$
$ \langle\frac{\partial}{\partial \bar{y}^{i}}, \frac{\partial}{\partial \bar{y}^{j}}\rangle_{g}=\delta_{i j},$
where $A=\left(a_{\beta}^{\alpha}\right) \in O(n+p)$ is an $(n+p) \times (n+p)$ orthogonal matrix.   Let \begin{equation}\label{nu-eq}\nu= \sum^{n+p}_{\theta=1}\nu_{\theta}\frac{\partial}{\partial\overline{y}^{\theta}}\in\mathbb{R}^{n+p},\end{equation} and $$g_{0\alpha\beta}=\langle\frac{\partial}{\partial\overline{y}^{\alpha}},\frac{\partial}{\partial\overline{y}^{\beta}}\rangle_{g_{0}}.$$
Let $w$ be a smooth function defined on the Riemannian manifold $\mathcal{M}^{n}$. Under the local coordinate system $\bar{y}=\left(\bar{y}^{1}, \cdots, \bar{y}^{n}\right)$, by an easy exercise, one can show that

\begin{equation}\label{n-ineq-2}\nu^{\top}= \sum^{n}_{\theta=1}\nu_{\theta}\frac{\partial}{\partial\overline{y}^{\theta}},\end{equation}and

\begin{equation}\begin{aligned}\label{n-ineq-1}\langle\nu, \nabla w\rangle_{g_{0}}=\sum^{n}_{i=1}\nu_{i}\frac{\partial w}{\partial \overline{y}^{i}} .\end{aligned}\end{equation}
By Cauchy-Schwarz inequality, we have

\begin{equation}\label{n-ineq-3}\begin{aligned}\left(\sum^{n}_{\theta=1}\nu_{\theta}\frac{\partial w}{\partial \overline{y}^{\theta}}\right)^{2}\leq\left(\sum^{n}_{\theta=1}\nu_{\theta}^{2}\right)\cdot\sum^{n}_{\theta=1}\left(\frac{\partial w}{\partial \overline{y}^{\theta}}\right)^{2}.\end{aligned}\end{equation}
Therefore, combining \eqref{n-ineq-2}, \eqref{n-ineq-1} and \eqref{n-ineq-3}, we can prove the following lemma.

\begin{lem}\label{new-lemma1}Let $w$ be a smooth function defined on the Riemannian manifold $\mathcal{M}^{n}$, then we have
\begin{equation}\label{ineq-nu-w}\langle\nu, \nabla w\rangle_{g_{0}}\leq|\nu^{\top}|_{g_{0}}|\nabla w|_{g}.\end{equation}\end{lem}
By a direct computation, one can show the following results of Chen and Cheng type.
\begin{lem}~\textbf{\emph{(Result of Cheng and Chen Type)}}\label{lem3.3}
Let $\left(y_{1}, \cdots, y_{n}\right)$ be an arbitrary coordinate system in a neighborhood $U$ of $P$ in $\mathcal{M}^{n} .$ Assume that $x$ with components $x_{\alpha}$ defined by
$
x_{\alpha}=x_{\alpha}\left(y_{1}, \cdots, y_{n}\right)$, where $1\leq \alpha \leq n+p$ is the position vector of $P$ in $\mathbb{R} ^{n+p}$.
Then, we have
\begin{equation}\label{v-2}
\sum_{\alpha=1}^{n+p}\left\langle\nabla x_{\alpha}, \nu\right\rangle_{g_{0}}^{2}=|\nu^{\top}|_{g_{0}}^{2},\end{equation}
where $\nabla$ is the gradient operator on $\mathcal{M}^{n}$.

\end{lem}
From Cauchy-Schwarz inequality, Lemma \ref{lem2.5} and Lemma \ref{lem3.3}, we have the following lemma.
\begin{lem}~\textbf{\emph{(Result of Cheng and Chen Type)}}\label{lem3.2}
Let $\left(y_{1}, \cdots, y_{n}\right)$ be an arbitrary coordinate system in a neighborhood $U$ of $P$ in $\mathcal{M}^{n} .$ Assume that $x$ with components $x_{\alpha}$ defined by $x_{\alpha}=x_{\alpha}\left(y_{1}, \cdots, y^{n}\right)$, where $1 \leq \alpha \leq n+p$,
is the position vector of $P$ in $\mathbb{R} ^{n+p}$.
Then, we have

\begin{equation}\label{uv-ineq}
\sum_{\alpha=1}^{n+p}\left\langle\nabla x_{\alpha}, \nabla u\right\rangle_{g}\left\langle\nabla x_{\alpha}, \nu\right\rangle_{g_{0}}\leq|\nabla u|_{g}|\nu^{\top}|_{g_{0}},\end{equation}
where $\nabla$ is the gradient operator on $\mathcal{M}^{n}$.

\end{lem}

Let $x_{1}, x_{2}, \cdots, x_{n+p}$ be the standard coordinate functions of $\mathbb{R}^{n+p}$ and
define an $((n+p) \times (n+p))$-matrix $D$ by
$D:=\left(d_{\alpha \beta}\right),~
{\rm  where}~ d_{\alpha \beta}=\int_{\Omega} x_{\alpha} u_{1} u_{\beta+1} .$  Using the orthogonalization of Gram and Schmidt, it is easy to see that there exist an upper triangle matrix $R=\left(R_{\alpha \beta}\right)$ and an orthogonal matrix $Q=\left(\tau_{\alpha \beta}\right)$ such that $R=QB,$ i.e.,
$
R_{\alpha \beta}=\sum_{\gamma=1}^{n+p} \tau_{\alpha \gamma} d_{\gamma \beta}=\int_{\Omega} \sum_{\gamma=1}^{n+p} \tau_{\alpha \gamma} x_{\gamma} u_{1} u_{\beta+1}=0,
$
for $1 \leq \beta<\alpha \leq n+p$. Defining \begin{equation}\label{h-a}h_{\alpha}=\sum_{\gamma=1}^{n+p} \tau_{\alpha \gamma} x_{\gamma},\end{equation} we have $\int_{\Omega} h_{\alpha} u_{1} u_{\beta+1}=0,$ where $1 \leq \beta<\alpha \leq n+p .$
Since $h_{\alpha}=\sum_{\gamma=1}^{n+p} \tau_{\alpha \gamma} x_{\gamma}$ and $Q$ is an orthogonal matrix, by
Lemma \eqref{lem2.5}, Lemma \eqref{lem3.3} and Lemma \eqref{lem3.2}, we can show the following lemma.

\begin{lem}Under the above convention, we have
\begin{equation}\label{n-2}
\sum_{\alpha=1}^{n+p}\left|\nabla h
_{\alpha}\right|_{g}^{2}=n,\end{equation}

\begin{equation}\label{de-h}\sum_{\alpha=1}^{n+p}\left(\Delta h_{\alpha}\right)^{2}=n^{2}H^{2},\end{equation}

\begin{equation}\label{hhu} \sum_{\alpha=1}^{n+p} \Delta h_{\alpha}\left\langle\nabla h_{\alpha}, \nabla u_{1}\right\rangle_{g}=0,
\end{equation}

\begin{equation}\label{hhv}
\sum_{\alpha=1}^{n+p} \Delta h_{\alpha}\left\langle\nabla h_{\alpha}, \nu\right\rangle_{g_{0}}=0,\end{equation}

\begin{equation}\label{vv2} \sum_{\alpha=1}^{n+p}\left\langle\nabla h_{\alpha}, \nu\right\rangle_{g_{0}}^{2}=\big|\nu^{\top}\big|_{g_{0}}^{2}, \end{equation}

\begin{equation}\label{huhv} \sum_{\alpha=1}^{n+p}\left\langle\nabla h_{\alpha}, \nabla u_{1}\right\rangle_{g}\left\langle\nabla h_{\alpha}, \nu\right\rangle_{g_{0}}\leq|\nabla u_{1}|_{g}|\nu^{\top}|_{g_{0}}, \end{equation}
and

\begin{equation}\label{nab-u-2} \sum_{\alpha=1}^{n+p}\left\langle\nabla h_{\alpha}, \nabla u_{1}\right\rangle_{g}^{2}=\left|\nabla u_{1}\right|_{g}^{2}. \end{equation}

\end{lem}

\vskip5mm

\section{Two Bounds for the Eigenvalues}\label{sec3}
\vskip3mm

In this section, applying general formulas, we give some bounds of the eigenvalues.

\subsection{Bound of Yang Type}\label{subsec3.1}

In this paper, we investigate the eigenvalues of Dirichlet problem \eqref{diri-prob} of Xin-Laplacian on the complete Riemannian manifolds. The first purpose of this paper is to prove an inequality of eigenvalues with higher order as follows.
\begin{thm}\label{thm1.1}
Let $(\mathcal{M}^{n},g)$ be an $n$-dimensional complete Riemannian manifold isometrically embedded into the Euclidean space $\mathbb{R}^{n+p}$ with mean curvature $H$.  Assume that $\Lambda_{i}$ denotes the $i$-th eigenvalue of the Dirichlet problem
\eqref{diri-prob}  of  the Xin-Laplacian. Then, we have

\begin{equation}
\begin{aligned}
\label{thm-1-11} \sum^{k}_{i=1}(\Lambda_{k+1}-\Lambda_{i})^{2}
\leq\frac4n\sum^{k}_{i=1}(\Lambda_{k+1}-\Lambda_{i}) \left(\Lambda_{i}+D_{1}\Lambda_{i}^{\frac{1}{2}}+\frac{1}{4}D_{1}^{2}+\frac{1}{4}C_{1}\right),
\end{aligned}
\end{equation}
and

\begin{equation}
\begin{aligned}
\label{thm-1-112} \sum^{k}_{i=1}(\Lambda_{k+1}-\Lambda_{i})^{2}
\leq\frac6n\sum^{k}_{i=1}(\Lambda_{k+1}-\Lambda_{i})\left(\Lambda_{i}+\frac{1}{2}D_{1}^{2}+\frac{1}{6}C_{1}\right),
\end{aligned}
\end{equation}
where $C_{1}=\inf_{\psi\in \Psi}\max_{\Omega}n^{2}H^{2}\ \ and\ \ D_{1}= \max_{\Omega} |\nu^{\top}|_{g_{0}}. $
\end{thm}

\begin{rem}In Theorem {\rm \ref{thm1.1}}, one can show that \eqref{thm-1-112} is sharper than \eqref{thm-1-11}, when
\begin{equation}
\begin{aligned}\label{ineq-1.11}
2\Lambda_{i}^{\frac{1}{2}}\left(\int_{\Omega}u_{i}^{2} |\nu^{\top}|^{2}_{g_{0}} e^{\langle\nu,X\rangle_{g_{0}}}dv\right)^{\frac{1}{2}}
\geq\Lambda_{i}+\int_{\Omega}u_{i}^{2}|\nu^{\top}|_{g_{0}}^{2}e^{\langle\nu,X\rangle_{g_{0}}}dv;
\end{aligned}
\end{equation} while \eqref{thm-1-11} is sharper than \eqref{thm-1-112}, when inequality \eqref{ineq-1.11} goes the other way.

\end{rem}
\begin{rem}In Theorem {\rm\ref{thm1.1}},  assuming that $|\nu^{\top}|_{g_{0}}=0$, one can deduce the following inequality:

\begin{equation*}
\begin{aligned} \sum^{k}_{i=1}(\Lambda_{k+1}-\Lambda_{i})^{2}
\leq\frac4n\sum^{k}_{i=1}(\Lambda_{k+1}-\Lambda_{i}) \left(\Lambda_i+\frac{1}{4}C_{1}\right),
\end{aligned}
\end{equation*}where $C_{1}=\inf_{\psi\in \Psi}\max_{\Omega}n^{2}H^{2},$  which is given by  Chen and Cheng in {\rm \cite{CC}}.
 \end{rem}

\begin{rem} For Theorem {\rm\ref{thm1.1}}, an analogous version with respect to the $\mathfrak{L}$ operator is obtained by Chen and Peng in {\rm \cite{CP}} and a similar result for the drifting Laplacian is obtained by Xia and Xu in {\rm \cite{XX}}. \end{rem}

 In this subsection, we give the proof of Theorem \ref{thm1.1}.\vskip 3mm
\noindent {\it Proof of Theorem {\rm \ref{thm1.1}}}.   From Nash's
Theorem, there exists an isometric immersion from $\mathcal{M}^{n}$ into the $(n+p)$-dimensional Euclidean space
$\mathbb{R}^{n+p}$. Assume $x_{1}, \cdots, x_{n+p}$ are $(n+p)$
coordinate functions of $\mathbb{R}^{n+p}$, then $x_{1}, \cdots,
x_{n+p}$ are defined on $\mathcal{M}^{n}$ globally. Taking  $\varphi=x_{\alpha}$,
for  $1 \leq\alpha\leq n+p$,  we have, from the Proposition
\ref{prop2.3},
\begin{equation*}
\begin{aligned}
\sum^{k}_{i=1}(\Lambda_{k+1}-\Lambda_{i})^{2}\|u_{i}\nabla
x_{\alpha}\|_{\Omega}^{2}
\leq\sum^{k}_{i=1}(\Lambda_{k+1}-\Lambda_{i}) \|2\langle\nabla
x_{\alpha},\nabla u_{i}\rangle_{g}+u_{i}\mathfrak{L}_{\nu}x_{\alpha}\|_{\Omega}^{2}.
\end{aligned}
\end{equation*}
Taking sum on $\alpha$ from 1 to $n+p$, we have
\begin{equation}
\begin{aligned}
\label{3.2}
\sum^{k}_{i=1}(\Lambda_{k+1}-\Lambda_{i})^{2}\sum^{n+p}_{\alpha=1}\|u_{i}\nabla
x_{\alpha}\|_{\Omega}^{2}
\leq\sum^{k}_{i=1}(\Lambda_{k+1}-\Lambda_{i})
\sum^{n+p}_{\alpha=1}\|2\langle\nabla x_{\alpha},\nabla
u_{i}\rangle_{g}+u_{i}\mathfrak{L}_{\nu}x_{\alpha}\|_{\Omega}^{2}.
\end{aligned}\end{equation}
Applying Lemma \ref{lem2.5}, Lemma \ref{lem3.3} and Lemma \ref{lem3.2}, we have
\begin{equation*}
\begin{aligned} &\sum^{n+p}_{\alpha=1}\|2\langle\nabla x_{\alpha},\nabla
u_{i}\rangle_{g}+u_{i}\mathfrak{L}_{\nu}x_{\alpha}\|_{\Omega}^{2}
\\&\leq4\Lambda_{i}+\int_{\Omega}u_{i}^{2}(n^{2}H^{2}+|\nu^{\top}|_{g_{0}}^{2})e^{\langle\nu,X\rangle_{g_{0}}}dv+4\int_{\Omega}u_{i}|\nabla u_{i}|_{g}|\nu|_{g_{0}} e^{\langle\nu^{\top},X\rangle_{g_{0}}}dv.
\end{aligned}
\end{equation*}Furthermore, by Cauchy-Schwarz inequality, we infer that
\begin{equation}
\begin{aligned}
\label{3.3} &\sum^{n+p}_{\alpha=1}\|2\langle\nabla x_{\alpha},\nabla
u_{i}\rangle_{g}+u_{i}\mathfrak{L}_{\nu}x_{\alpha}\|_{\Omega}^{2}
\\ &\leq4\Lambda_{i}+\int_{\Omega}u_{i}^{2}(n^{2}H^{2}+|\nu^{\top}|_{g_{0}}^{2})e^{\langle\nu,X\rangle_{g_{0}}}dv\\&\quad+4\left(\int_{\Omega}|\nabla u_{i}|^{2}_{g} e^{\langle\nu,X\rangle_{g_{0}}}dv\right)^{\frac{1}{2}}\left(\int_{\Omega}u_{i}^{2} |\nu^{\top}|^{2}_{g_{0}} e^{\langle\nu,X\rangle_{g_{0}}}dv\right)^{\frac{1}{2}}
\\ &=4\Lambda_{i}+\int_{\Omega}u_{i}^{2}(n^{2}H^{2}+|\nu^{\top}|_{g_{0}}^{2})e^{\langle\nu,X\rangle_{g_{0}}}dv+4\Lambda_{i}^{\frac{1}{2}}\left(\int_{\Omega}u_{i}^{2} |\nu^{\top}|^{2}_{g_{0}} e^{\langle\nu,X\rangle_{g_{0}}}dv\right)^{\frac{1}{2}}.
\end{aligned}
\end{equation}
By the mean value inequality, we have
\begin{equation}
\begin{aligned}
\label{3.3-1} \sum^{n+p}_{\alpha=1}\|2\langle\nabla x_{\alpha},\nabla
u_{i}\rangle_{g}+u_{i}\mathfrak{L}_{\nu}x_{\alpha}\|_{\Omega}^{2}
 \leq 6\Lambda_{i}+\int_{\Omega}u_{i}^{2}(n^{2}H^{2}+3|\nu^{\top}|_{g_{0}}^{2})e^{\langle\nu,X\rangle_{g_{0}}}dv.
\end{aligned}
\end{equation}
Therefore, from \eqref{3.2}, \eqref{3.3}, \eqref{3.3-1}  and Lemma \ref{lem2.5},
we infer that
\begin{equation*}
\begin{aligned}
&\sum^{k}_{i=1}(\Lambda_{k+1}-\Lambda_{i})^{2}\\
&\leq\frac{4}{n}\sum^{k}_{i=1}(\Lambda_{k+1}-\Lambda_{i})\\&\times
\left(\Lambda_{i}+\frac{1}{4}\int_{\Omega}u_{i}^{2}\left(n^{2}H^{2} +|\nu^{\top}|_{g_{0}}^{2}\right)e^{\langle\nu,X\rangle_{g_{0}}}dv+\Lambda_{i}^{\frac{1}{2}}\left(\int_{\Omega}u_{i}^{2} |\nu^{\top}|^{2}_{g_{0}} e^{\langle\nu,X\rangle_{g_{0}}}dv\right)^{\frac{1}{2}}\right),
\end{aligned}
\end{equation*}and

\begin{equation*}
\begin{aligned}
&\sum^{k}_{i=1}(\Lambda_{k+1}-\Lambda_{i})^{2}
 \leq\frac{6}{n}\sum^{k}_{i=1}(\Lambda_{k+1}-\Lambda_{i})
\left(\Lambda_{i}+\frac{1}{6}\int_{\Omega}u_{i}^{2}\left(n^{2}H^{2} +3|\nu^{\top}|_{g_{0}}^{2}\right)e^{\langle\nu,X\rangle_{g_{0}}}dv\right).
\end{aligned}
\end{equation*}
Since eigenvalues are invariant in the sense of isometries, defining
$C_{1}=\inf_{\psi\in \Psi}\max_{\Omega}n^{2}H^{2}$, and $D_{1}= \max_{\Omega}|\nu^{\top}|_{g_{0}}$,
where $\Psi$ denotes the set of all isometric immersions from $\mathcal{M}^n$
into a Euclidean space, we have
\begin{equation*}
\begin{aligned}
&\sum^{k}_{i=1}(\Lambda_{k+1}-\Lambda_{i})^{2}\leq\frac{4}{n}\sum^{k}_{i=1}(\Lambda_{k+1}-\Lambda_{i})
\left(\Lambda_{i}+D_{1}\Lambda_{i}^{\frac{1}{2}}+\frac{1}{4}D_{1}^{2}+\frac{1}{4}C_{1}\right),
\end{aligned}
\end{equation*} and
\begin{equation*}
\begin{aligned}
&\sum^{k}_{i=1}(\Lambda_{k+1}-\Lambda_{i})^{2}\leq\frac{6}{n}\sum^{k}_{i=1}(\Lambda_{k+1}-\Lambda_{i})
\left(\Lambda_{i}+\frac{1}{2}D_{1}^{2}+\frac{1}{6}C_{1}\right),
\end{aligned}
\end{equation*}as claimed.
It finishes the proof of Theorem \ref{thm1.1}.
\begin{flushright}
$\square$
\end{flushright}

\noindent
By observing the proof of the Theorem\ref{thm1.1}, one has the following corollary.
\begin{corr}\label{corr3.1}
For an $n$-dimensional  complete Riemannian manifold $\mathcal{M}^{n}$, there
exists a function $H$ such that eigenvalues $\Lambda_{i}$ of the
Dirichlet  eigenvalue problem \eqref{diri-prob} of the differential operator $\mathfrak{L}_{\nu}$
satisfy
\begin{equation}\label{3.4}
\begin{aligned}
\sum^{k}_{i=1}(\Lambda_{k+1}-\Lambda_{i})^{2}
&\leq\frac{4}{n}\sum^{k}_{i=1}(\Lambda_{k+1}-\Lambda_{i})\\&\times
\left(\Lambda_{i}+\frac{1}{4}\int_{\Omega}u_{i}^{2}\left(n^{2}H^{2} +|\nu^{\top}|_{g_{0}}^{2}+4\Lambda_{i}^{\frac{1}{2}}|\nu^{\top}|_{g_{0}}\right)e^{\langle\nu,X\rangle_{g_{0}}}dv\right),
\end{aligned}
\end{equation}and
\begin{equation}\label{3.4-1}
\begin{aligned}
\sum^{k}_{i=1}(\Lambda_{k+1}-\Lambda_{i})^{2}
&\leq\frac{6}{n}\sum^{k}_{i=1}(\Lambda_{k+1}-\Lambda_{i})\\&\times
\left(\Lambda_{i}+\frac{1}{4}\int_{\Omega}u_{i}^{2}\left(n^{2}H^{2} +3|\nu^{\top}|_{g_{0}}^{2}\right)e^{\langle\nu,X\rangle_{g_{0}}}dv\right).
\end{aligned}
\end{equation}
\end{corr}
\subsection{Bound of Payne-P\'{o}lya-Weinberger Type}\label{subsec3.2}

In this subsection, we shall give a bound of Payne-P\'{o}lya-Weinberger type for the Xin-Laplacian as follows.
\begin{thm}
\label{thm1.2}
Let $\mathcal{M}^{n}$ be an $n$-dimensional complete Riemannian
manifold isometrically embedded into the Euclidean space $\mathbb{R}^{n+p}$ with mean curvature $H$. Then,  for any $j(j=1,2,\cdots)$, Dirichlet  problem {\rm \eqref{diri-prob}}
of the Xin-Laplacian satisfy

\begin{equation}
\begin{aligned}
\label{z-ineq-2} \sum_{l=1}^{n}\Lambda_{j+l}\leq
(4+n)\Lambda_{i}+4D_{1}\Lambda_{i}^{\frac{1}{2}}+D_{1}^{2}+C_{1},
\end{aligned}
\end{equation}and
\begin{equation}
\begin{aligned}
\label{z-ineq-2-1} \sum_{l=1}^{n} \Lambda_{j+l} \leq
(6+n)\Lambda_{i}+3D_{1}^{2}+C_{1},
\end{aligned}
\end{equation}
where $C_{1}=\inf_{\psi\in \Psi}\max_{\Omega}n^{2}H^{2}\ \ and \ \ D_{1}= \max_{\Omega} |\nu^{\top}|_{g_{0}}.$
\end{thm}

\begin{rem}In Theorem {\rm \ref{thm1.2}}, when $|\nu^{\top}|_{g_{0}}=0$, for all $j=1,2,\cdots,$ we have

\begin{equation*}
\frac{\Lambda_{j+1} +\Lambda_{j+2} +\cdots+ \Lambda_{j+n}}{
\Lambda_{j}}\leq n + 4,\end{equation*} which generalizes   Ashbaugh and Benguria's universal inequality \eqref{1.16}.\end{rem}

\noindent
{\it Proof of Theorem {\rm \ref{thm1.2}}}.   Nash's
Theorem implies that there exists an isometric immersion from
$\mathcal{M}^{n}$ into $\mathbb{R}^{n+p}$. Let $x_{1}, \cdots, x_{n+p}$ be
coordinate functions of $\mathbb{R}^{n+p}$. Then  $x_{1}, \cdots,
x_{n+p}$ are defined on $\mathcal{M}^{n}$ globally. Applying the Proposition
\ref{prop2.2} to functions $\phi_l=x_{l}$, we obtain
\begin{equation}
\begin{aligned}\label{4.1}
\sum^{n+p}_{l=1}(\Lambda_{j+l}-\Lambda_{j})\|u_j\nabla
\Phi_{l}\|^{2}_{\Omega} \leq
\sum^{n+p}_{l=1}\int_{\Omega}\big(u_{j}\mathfrak{L}_{\nu}\Phi_{l}
+2\langle\nabla \Phi_{l},\nabla u_{j}\rangle_{g}\big)^{2}e^{\langle\nu,X\rangle_{g_{0}}}dv,
\end{aligned}
\end{equation}
with $\Phi_{l}=\sum_{s=1}^{n+p}a_{ls}x_{s},$ where $A=(a_{ij})_{(n+p)\times(n+p)}$ is an orthogonal matrix.
Furthermore, we know that $\Phi_{l}$ satisfies Proposition
\ref{prop2.2} since $A=(a_{lt})$ is an
orthogonal matrix. By an orthogonal transformation, it is not hard
to prove, for any $l$, $|\nabla \Phi_{l}|_{g}^{2}\leq 1.$
Furthermore, with a simple calculation, we derive

\begin{equation*}
\begin{aligned}
&\quad~\sum^{n+p}_{l=1}(\Lambda_{j+l}-\Lambda_{j})\|\nabla \Phi_{l}u_{j}\|^{2}_{\Omega}\\
&\geq\sum^{n}_{l=1}(\Lambda_{j+l}-\Lambda_{j})\|\nabla
\Phi_{l}u_{j}\|^{2}_{\Omega}
+(\Lambda_{j+n+1}-\Lambda_{j})\sum^{n+p}_{l=n+1}\|\nabla
\Phi_{l}u_{j}\|^{2}_{\Omega}\\
&=\sum^{n}_{l=1}(\Lambda_{j+l}-\Lambda_{j})\|\nabla
\Phi_{l}u_{j}\|^{2}_{\Omega}
+(\Lambda_{j+n+1}-\Lambda_{j})\int_{\Omega}\left(\sum_{l=1}^{n}(1-|\nabla
\Phi_{l}|_{g}^{2})\right)u_{j}^{2}e^{\langle\nu,X\rangle_{g_{0}}}dv\\
&\geq\sum^{n}_{l=1}(\Lambda_{j+l}-\Lambda_{j})\|\nabla
\Phi_{l}u_{j}\|^{2}_{\Omega}
+\sum^{n}_{l=1}(\Lambda_{j+l}-\Lambda_{j})\int_{\Omega}\left(u_{1}^{2}-|\nabla
\Phi_{l}|_{g}^{2}\right)u_{j}^{2}e^{\langle\nu,X\rangle_{g_{0}}}dv,
\end{aligned}
\end{equation*}
which tells us
\begin{equation}
\begin{aligned}
\label{3.10}\sum_{l=1}^{n}(\Lambda_{j+l}-\Lambda_{j}) \leq
\int_{\Omega}\sum^{n+p}_{l=1}(u_{j}\mathfrak{L}_{\nu}\Phi_{l} +2\langle\nabla
\Phi_{l},\nabla u_{j}\rangle_{g})^{2}e^{\langle\nu,X\rangle_{g_{0}}}dv.
\end{aligned}
\end{equation}From Lemma \ref{lem2.5}, we have
\begin{equation}
\begin{aligned}
\label{3.11}\sum^{n+p}_{l=1}\left(\Delta\Phi_{l}\right)^{2}=n^{2}H^{2},
\end{aligned}
\end{equation}

\begin{equation}\label{3.12}
\begin{aligned}\sum^{n+p}_{l=1}\Delta\Phi_{l}\nabla\Phi_{l}=0,\end{aligned}
\end{equation} and

\begin{equation}\label{3.13}
\begin{aligned}\sum^{n+p}_{l=1}\langle\nabla\Phi_{l},\nabla w\rangle_{g}\langle\nabla\Phi_{l},\nabla v\rangle_{g}=\langle\nabla w,\nabla v\rangle_{g},
\end{aligned}
\end{equation}since $A=(a_{lt})$ is an $(n+p)\times(n+p)$-orthogonal matrix.
Substituting \eqref{3.11},  \eqref{3.12}, \eqref{3.13} into \eqref{3.10}, we infer that

\begin{equation*}
\begin{aligned}
&\sum_{l=1}^{n}(\Lambda_{j+l}-\Lambda_{j})\\&\leq
\int_{\Omega}\sum^{n+p}_{l=1}\left\{u_{j}\mathfrak{L}_{\nu}\Phi_{l} +2\langle\nabla
\Phi_{l},\nabla u_{j}\rangle_{g}\right\}^{2}e^{\langle\nu,X\rangle_{g_{0}}}dv
\\&\leq4\Lambda_{i}+
\int_{\Omega}\Big{\{}u_{j}^{2}(n^{2}H^{2}+|
\nu^{\top}|_{g_{0}}^{2})\Big{\}}e^{\langle\nu,X\rangle_{g_{0}}}dv +4\int_{\Omega}u_{i}|\nabla u_{j}|_{g}|\nu|_{g_{0}}e^{\langle\nu,X\rangle_{g_{0}}}dv.
\end{aligned}
\end{equation*}
Therefore, by Cauchy-Schwarz inequality and the mean value inequality, we have
\begin{equation*}
\begin{aligned}
\sum_{l=1}^{n}(\Lambda_{j+l}-\Lambda_{j})\leq
4\Lambda_{i}+\int_{\Omega}u_{j}^{2}\left(n^{2}H^{2} +|\nu^{\top}|_{g_{0}}^{2}+4\Lambda_{i}^{\frac{1}{2}}|\nu^{\top}|_{g_{0}}\right)e^{\langle\nu,X\rangle_{g_{0}}}dv\\
\end{aligned}
\end{equation*}
and

\begin{equation*}
\begin{aligned}
\sum_{l=1}^{n}(\Lambda_{j+l}-\Lambda_{j})\leq
6\Lambda_{i}+\int_{\Omega}u_{j}^{2}\left(n^{2}H^{2} +3|\nu^{\top}|_{g_{0}}^{2}\right)e^{\langle\nu,X\rangle_{g_{0}}}dv,
\end{aligned}
\end{equation*}as desired.
The proof of the Theorem \ref{thm1.2} is finished.
\begin{flushright}
$\square$
\end{flushright}

\begin{corr}
\label{corr3.3} For  an $n$-dimensional complete Riemannian manifold
$\mathcal{M}^{n}$, there exists a function $H$ such that, for any $j=1,2,\cdots$,
 eigenvalues of the Dirichlet eigenvalue problem \eqref{diri-prob} of the differential operator $\mathfrak{L}_{\nu}$
satisfy
\begin{equation*}
\begin{aligned}
\sum_{l=1}^{n}\Lambda_{j+l}\leq
(4+n)\Lambda_{i}+\int_{\Omega}u_{j}^{2}\left(n^{2}H^{2} +|\nu^{\top}|_{g_{0}}^{2}+4\Lambda_{i}^{\frac{1}{2}}|\nu^{\top}|_{g_{0}}\right)e^{\langle\nu,X\rangle_{g_{0}}}dv,
\end{aligned}
\end{equation*}
and

\begin{equation*}
\begin{aligned}
\sum_{l=1}^{n} \Lambda_{j+l} \leq
(6+n)\Lambda_{i}+\int_{\Omega}u_{j}^{2}\left(n^{2}H^{2} +3|\nu^{\top}|_{g_{0}}^{2}\right)e^{\langle\nu,X\rangle_{g_{0}}}dv.
\end{aligned}
\end{equation*}
\end{corr}

\begin{rem}In  some of the most important early literatures {\rm \cite{Ha,HM,HS}}, Harrell II,  Michel,  Stubbe developed an algebraic technique to discuss the eigenvalue problems and this technique enable them to explore the universal inequality  in various of settings. Recently, based on an algebraic technique, some similar eigenvalue inequalities of Schr\"{o}dinger operator  are established by Soufi,  Harrell  and  Ilias in their cerebrated paper {\rm \cite{SHI}}. However, the method estimating the eigenvalues in this paper is different from their one.
\end{rem}

\subsection{A Remark on Theorem \ref{thm1.1} and Theorem \ref{thm1.2} }  According to Colin de Verdi\`{e}re's construction {\rm \cite{CV}} and the celebrated  isometric embedding theorem due to Nash and Moser, there exist no universal inequalities for the eigenvalues of the Laplace operator on the bounded domain of a Riemannian submanifold isometrically embedded into the Euclidean space, unless it is a submanifold with constant mean curvature. Likewise, for some submanifolds of the Euclidean spaces, eigenvalues of the Xin-Laplacian do not satisfy universal inequalities. Therefore, in the absence of any other condition of intrinsic geometry, the spectrum of drifting Laplacian naturally contains information about the extrinsic geometry of a submanifold when it is embedded into certain Euclidean space. However, for translating solitons, one can obtain some universal inequalities. For example, see Section {\rm \ref{sec4}}.

\section{Eigenvalues on the Translating Solitons}\label{sec4}
In this section, we would like to exploit the eigenvalues of Xin-Laplacian on the translating solitons.

\subsection{Translating Solitons Associated with MCF} \label{subsec4.1}
Let $X : \mathcal{M}^{n}\rightarrow \mathbb{R}^{n+p}$ be an isometric immersion from an $n$-dimensional, oriented, complete
Riemannian manifold $\mathcal{M}^{n}$ to the Euclidean space $\mathbb{R}^{n+p}$.  We consider a smooth family of immersions
$X_{t} = X(\cdot,t):\mathcal{M}^{n}\rightarrow \mathbb{R}^{n+p}$ with corresponding images $\mathcal{M}^{n}_{t} = X_{t}(\mathcal{M}^{n})$ such that the following mean curvature equation system \cite{H,LXX}:

\begin{equation}\label{MCF-Equa}{\begin{cases}
&\frac{d}{dt}X(x,t)=\textbf{H}(x,t), x\in \mathcal{M}^{n}   \\
& X(\cdot,0) = X(\cdot),
\end{cases}}\end{equation}
is satisfied, where $\textbf{H}(x,t)$ is the mean curvature vector of $\mathcal{M}_{t}$ at $X(x, t)$ in $\mathbb{R}^{n+p}$.  We let $\nu_{0}$ be a constant vector with unit length in $\mathbb{R}^{n+p}$. We  denote $\nu_{0}^{\bot}$ the normal projection
of $\nu_{0}$ to the normal bundle of $\mathcal{M}^{n}$ in $\mathbb{R}^{n+p}$. A submanifold $X:\mathcal{M}^{n}\rightarrow\mathbb{R}^{n+p}$ is said to be a translating soliton of the mean curvature flow \eqref{MCF-Equa}, if it
satisfies:
$ \textbf{H}=\nu_{0}^{\bot},$ which is a special solution of the mean curvature flow equation \eqref{MCF-Equa}.
 They are not only special solutions to the mean curvature
flow equations, but they often occur as Type-II singularity of a mean curvature flow, which play an important role in the study
of the mean curvature flow \cite{AV}.
In \cite{Xin2}, Xin studied some basic properties of translating solitons: the volume growth, generalized maximum principle, Gauss maps and certain functions related to the Gauss maps. In addition, he carried out point-wise estimates and integral estimates for the squared norm of the second fundamental form. By utilizing these estimates,  Xin proved some rigidity theorems for translating solitons in the Euclidean space in higher codimension in \cite{Xin2}. Recently, Chen and Qiu \cite{ChQ} proved established a nonexistence theorem for the spacelike translating solitons. These results are proved by using a new Omori-Yau maximal principle. To agree with the notation appearing in \cite{Xin2}, we denote $\mathfrak{L}_{\nu_{0}}$ by $\mathfrak{L}_{II}$ henceforth in the section.

\subsection{Eigenvalue Inequality of Yang Type}\label{subsec4.2}
As an application of Theorem \ref{thm1.1}, we study the eigenvalues of $\mathfrak{L}_{II}$ operator on the complete translating solitons in this section. More precisely, we prove the following theorem.
\begin{thm}\label{thm5.1}
Let  $\mathcal{M}^{n}$ be an $n$-dimensional  complete  translating soliton isometrically embedded into the Euclidean space $\mathbb{R}^{n+p}$. Then,
the eigenvalues $\Lambda_{i}(1\leq i\leq k)$ of
Dirichlet eigenvalue problem \eqref{diri-prob} of the differential operator $\mathfrak{L}_{II}$
satisfy
\begin{equation}\label{inq-4.1}
\begin{aligned}
\sum^{k}_{i=1}(\Lambda_{k+1}-\Lambda_{i})^{2}
\leq\frac{4}{n}\sum^{k}_{i=1}(\Lambda_{k+1}-\Lambda_{i})
\left(\Lambda_{i}+\Lambda_{i}^{\frac{1}{2}}+\frac{n^{2}}{4} \right),
\end{aligned}
\end{equation}and, for any $n\geq2$,

\begin{equation}\label{3.5-1}
\begin{aligned}
\sum^{k}_{i=1}(\Lambda_{k+1}-\Lambda_{i})^{2}
\leq\frac{6}{n}\sum^{k}_{i=1}(\Lambda_{k+1}-\Lambda_{i})
\left(\Lambda_{i}+\frac{n^{2}}{6}\right).
\end{aligned}
\end{equation}
\end{thm}

\begin{proof}Since $\mathcal{M}^{n}$ is an $n$-dimensional  complete  translating soliton isometrically embedded into the $(n+p)$ dimensional Euclidean space $\mathbb{R}^{n+p}$, we have

\begin{equation}\label{3.6}H=\nu_{0}^{\perp} \ \ {\rm and} \ \ |\nu_{0}|_{g_{0}}^{2}=1,\end{equation}  which implies that \begin{equation}\label{3.7}\int_{\Omega}n^{2}H^{2}e^{\langle\nu_{0},X\rangle_{g_{0}}}dv= \int_{\Omega}n^{2}|\nu_{0}^{\perp}|^{2}e^{\langle\nu_{0},X\rangle_{g_{0}}}dv\leq n^{2}.\end{equation}
Combining with  \eqref{3.6} and \eqref{3.7} yields

\begin{equation}\label{3.8}\frac{1}{4}\int_{\Omega}u_{i}^{2}\left(n^{2}H^{2} +|\nu_{0}^{\top}|_{g_{0}}^{2}\right)e^{\langle\nu_{0},X\rangle_{g_{0}}}dv\leq \frac{n^{2}}{4}.\end{equation}
 Substituting \eqref{3.8} into \eqref{3.4}, we obtain \eqref{inq-4.1}. The proof of \eqref{3.5-1} is similar.\end{proof}

\subsection{Eigenvalue Inequality of Levitin-Parnovski Type}\label{subsec4.3}

Applying Theorem \ref{thm1.2}, we can show the following theorem.
\begin{thm}\label{thm5.2}
Let  $\mathcal{M}^{n}$ be an $n$-dimensional  complete  translating soliton isometrically embedded into the Euclidean space $\mathbb{R}^{n+p}$. Then, for any $j=1,2,\cdots$,
the eigenvaluesof
Dirichlet eigenvalue problem \eqref{diri-prob} of the differential operator $\mathfrak{L}_{II}$
satisfy

\begin{equation*}
\begin{aligned}
  \sum^{n}_{k=1}\Lambda_{j+k} \leq(n+4)\Lambda_{j} +n^{2}+4\Lambda_{i}^{\frac{1}{2}},
\end{aligned}
\end{equation*}and, for any $n\geq2$,

\begin{equation*}
\begin{aligned}
  \sum^{n}_{k=1}\Lambda_{j+k} \leq(n+6)\Lambda_{j} +n^{2}.
\end{aligned}
\end{equation*}
\end{thm}

\begin{proof}The proof is similar to Theorem \ref{thm5.1}. Therefore, we omit it here.

\end{proof}

\begin{rem}
Roughly speaking, it is very difficult to obtain universal inequalities of Witten-Laplacian on the Ricci solitons in the sense of Ricci flows or the self-shrinkers in the sense of the mean curvature flows unless it is a trivial Ricci soliton and there are some special assumption for the potential function $f$. For example, see {\rm\cite{XX}}.  However, it is surprising that, in Theorem {\rm \ref{thm5.2}} and  Theorem {\rm \ref{thm5.1}}, the eigenvalue inequalities are universal.

\end{rem}

\subsection{Estimates for the Upper Bounds and Gaps of Consecutive Eigenvalues }\label{subsec4.4}
In what follows, we will give several applications of Theorem {\rm \ref{thm5.2}} and  Theorem {\rm \ref{thm5.1}}. First of all,  by \eqref{inq-4.1}, we have\begin{equation}\label{thm0-1-11}n \sum_{i=1}^{k}\left(\Lambda_{k+1}-\Lambda_{i}\right)^{2} \leq \sum_{i=1}^{k}\left(\Lambda_{k+1}-\Lambda_{i}\right)\left(4 \Lambda_{i}+4\Lambda^{\frac{1}{2}}_{i}+n^{2}\right).\end{equation}
Since the formula \eqref{thm0-1-11} is a quadratic inequality of $\Lambda_{k+1}$, according to the direct but somewhat tedious calculation,   one can get

\begin{equation*}\begin{aligned} \Lambda_{k+1}&\leq
 \frac{1}{k}\sum_{i=1}^{k}\left[\frac{2}{n} \left(\Lambda_{i}+\Lambda^{\frac{1}{2}}_{i}\right)+\Lambda_{i}\right]+\frac{n}{2} \\&+ \Bigg{\{}\left[\frac{1}{k}\frac{2}{n}\sum_{i=1}^{k} \left(\Lambda_{i}+\Lambda^{\frac{1}{2}}_{i}\right)+\frac{n}{2}\right]^{2}-\left(1+\frac{4}{n}\right) \frac{1}{k} \sum_{i=1}^{k}\left(\Lambda_{i}-\frac{1}{k} \sum_{j=1}^{k} \Lambda_{j}\right)^{2}
\\&+\frac{1}{k}\frac{4}{n}\left[\sum_{i=1}^{k}\Lambda^{\frac{1}{2}}_{i}\left(\frac{1}{k}\sum_{j=1}^{k}\Lambda_{j} -\Lambda_{i}\right)\right]\Bigg{\}}^{\frac{1}{2}}.\end{aligned}\end{equation*}

Thus, we have the following estimates for the upper bound of the eigenvalues of Xin-Laplacian on the translating solitons.
\begin{corr}
For an $n$-dimensional complete translating soliton $(\mathcal{M}^{n},g)$, the
$k^{\text{th}}$ eigenvalue $\Lambda_{k}$ of the Dirichlet eigenvalue
problem \eqref{diri-prob} of the
differential operator $\mathfrak{L}_{II}$ satisfy,

\begin{equation}\begin{aligned}\label{inq-4.11}\Lambda_{k+1}&\leq
 \frac{1}{k}\sum_{i=1}^{k}\left[\frac{2}{n} \left(\Lambda_{i}+\Lambda^{\frac{1}{2}}_{i}\right)+\Lambda_{i}\right]+\frac{n}{2} + \Bigg\{\left[\frac{1}{k}\frac{2}{n}\sum_{i=1}^{k} \left(\Lambda_{i}+\Lambda^{\frac{1}{2}}_{i}\right)+\frac{n}{2}\right]^{2}\\&-\left(1+\frac{4}{n}\right) \frac{1}{k} \sum_{i=1}^{k}\left(\Lambda_{i}-\frac{1}{k} \sum_{j=1}^{k} \Lambda_{j}\right)^{2}
 +\frac{1}{k^{2}}\frac{4}{n} \sum_{i,j=1}^{k} \left(  \Lambda_{i}^{\frac{1}{2}}\Lambda_{j} -\Lambda_{i}^{\frac{3}{2}}\right) \Bigg\}^{\frac{1}{2}}.\end{aligned}\end{equation}

\end{corr}
If we use a positive real-valued function $f(n)$ to replace $n$ in the proofs of Theorem 2.1 in \cite{CY4},   then we can extend Cheng and Yang's recursion formula  to the following general case.

\vskip 2mm
\noindent

\begin{thm}\label{cy-type-re}{\bf( A recursion formula of Cheng and Yang Type)}.  Let  $\mu_1 \leq  \mu_2 \leq  \dots,
\leq \mu_{k+1}$ be any positive  real numbers satisfying
\begin{equation*}
  \sum_{i=1}^k(\mu_{k+1}-\mu_i)^2 \le
 \frac{4}{f(n)}\sum_{i=1}^k\mu_i(\mu_{k+1} -\mu_i).
\end{equation*}
 Define
 \begin{equation*}
 \Lambda_k=\frac 1k\sum_{i=1}^k\mu_i,\qquad T_k=\frac 1k
\sum_{i=1}^k\mu_i^2, \ \ \
F_k=\left (1+\frac{2}{f(n)} \right )\Lambda_k^2-T_k.
\end{equation*}
Then, we have
\begin{equation*}
F_{k+1}\leq C(n,k) \left ( \frac {k+1}k \right )^{\frac{4}{f(n)}}F_k,
\end{equation*}
where
$$
C(n,k) =1-\frac 1{3f(n)}
  \left (\frac k{k+1}\right )^{\frac{4}{f(n)}}\frac {\left(1+\frac{2}{f(n)}\right )\left (1+
  \frac{4}{f(n)}\right )}{(k+1)^3}<1.
$$\end{thm}

\vskip 2mm
\noindent

By Theorem \ref{cy-type-re}, we can show the following proposition. See \cite{CZe}.
\begin{prop}\label{prop-5.6} Let $\lambda_{1} \leq \lambda_{2} \leq \cdots \leq \lambda_{k+1}$ be any positive real numbers satisfying the following inequality

\begin{equation}\label{gen-cas-f}
\sum_{i=1}^{k}\left(\lambda_{k+1}-\lambda_{i}\right)^{2} \leq \frac{4}{f(n)} \sum_{i=1}^{k} \lambda_{i}\left(\lambda_{k+1}-\lambda_{i}\right),
\end{equation}with $f(n)>0$.
 Then we have

\begin{equation}\label{Z-co}
\lambda_{k+1} \leq \left(1+\frac{4}{f(n)}\right) k^{\frac{2}{f(n)}} \lambda_{1}.
\end{equation} \end{prop}
\begin{proof}By using the same approach in \cite{CY4}, we can give the proof of this proposition. Here, we leave out the details of this proof.\end{proof}

Taking $f(n)=\frac{2 n}{3}$ and $\lambda_{i}=\Lambda_{i}+\frac{n^{2}}{6}$ in \eqref{gen-cas-f}, we can give an estimate for the upper bound of  the eigenvalues of $\mathfrak{L}_{II}$
on the translating solitons.

\begin{corr}\label{coro-5.7}
For an $n$-dimensional complete translating soliton $(\mathcal{M}^{n},g)$, the
$k^{\text{th}}$ eigenvalue $\Lambda_{k}$ of the Dirichlet eigenvalue
problem \eqref{diri-prob} of the
differential operator $\mathfrak{L}_{II}$ satisfy, for any
$k\geq 1$,
\begin{equation}\label{cy2}
\Lambda_{k+1}+ \frac{n^{2}}{6}
\leq \left(1 +\frac{6}{n}\right) \left(\Lambda_{1}+\frac{n^{2}}{6}\right) \ k^{\frac{3}{n}}.
\end{equation}

\end{corr}
\begin{rem}The upper bound \eqref{cy2} is not sharp in the sense of order of eigenvalues. Recall that, in {\rm \cite{CY4}}, by establishing a recursion formula, Cheng and Yang gave a sharp upper bound in the sense of order of eigenvalues.
Unfortunately, Cheng and Yang's recursion formula dose not work in our situation. In other words, Proposition {\rm \ref{prop-5.6}} can not apply directly to inequality \eqref{inq-4.1}. Therefore, we are fail to get a sharp upper bound of Cheng and Yang type. To get a sharp upper bound, it seems that a similar recursion formula needs to be proved.

\end{rem}
Next, we give an estimate for the upper bound of the gap of consecutive eigenvalues of the Dirichlet problem \eqref{diri-prob} of $\mathfrak{L}_{II}$ operator on the translating solitons.

\begin{corr}\label{cor5.10}Under the same condition as Theorem {\rm \ref{thm5.1}}, we have

\begin{equation}\begin{aligned}\label{gap-product-1}
&\Lambda_{k+1}-\Lambda_{k} \leq 2 \left\{\left[\frac{1}{k}\frac{2}{n}\sum_{i=1}^{k} \left(\Lambda_{i}+\Lambda^{\frac{1}{2}}_{i}\right)+\frac{n}{2}\right]^{2}\right.\\&\left.-\left(1+\frac{4}{n}\right) \frac{1}{k} \sum_{i=1}^{k}\left(\Lambda_{i}-\frac{1}{k} \sum_{j=1}^{k} \Lambda_{j}\right)^{2}
 +\frac{1}{k^{2}}\frac{4}{n} \sum_{i,j=1}^{k} \left(  \Lambda_{i}^{\frac{1}{2}}\Lambda_{j} -\Lambda_{i}^{\frac{3}{2}}\right)\right\}^{\frac{1}{2}},
\end{aligned}
\end{equation} and

\begin{equation}\begin{aligned}\label{gap-product-2}\Lambda_{k+1}&-\Lambda_{k} \leq 2\left[\left(\frac{3}{n} \frac{1}{k} \sum_{i=1}^{k} \Lambda_{i}+\frac{n}{3}\right)^{2} -\left(1+\frac{6}{n}\right) \frac{1}{k} \sum_{j=1}^{k}\left(\Lambda_{j}-\frac{1}{k} \sum_{i=1}^{k} \Lambda_{i}\right)^{2}\right]^{\frac{1}{2}}.\end{aligned}
\end{equation}
\end{corr}
\begin{proof} The strategy of this proof is similar to Corollary 1 in \cite{CY1}. For completeness, we give the proof of this corollary.
Since $k$ is an any integer, we know that \eqref{inq-4.11} is also true if we replace $k+1$ with $k$. Equivalently, we have
$$
n \sum_{i=1}^{k-1}\left(\Lambda_{k}-\Lambda_{i}\right)^{2} \leq \sum_{i=1}^{k-1}\left(\Lambda_{k}-\Lambda_{i}\right)\left(4 \Lambda_{i}+4\Lambda_{i}^{\frac{1}{2}}+n^{2}\right).
$$
Namely, $\Lambda_{k}$ also satisfies the same quadratic inequality. Therefore, we infer

\begin{equation}\begin{aligned}\label{inq-4.20} \Lambda_{k+1}&\geq
 \frac{1}{k}\sum_{i=1}^{k}\left[\frac{2}{n} \left(\Lambda_{i}+\Lambda^{\frac{1}{2}}_{i}\right)+\Lambda_{i}\right]+\frac{n}{2}+ \Bigg\{\left[\frac{1}{k}\frac{2}{n}\sum_{i=1}^{k} \left(\Lambda_{i}+\Lambda^{\frac{1}{2}}_{i}\right)+\frac{n}{2}\right]^{2} \\&-\left(1+\frac{4}{n}\right) \frac{1}{k} \sum_{i=1}^{k}\left(\Lambda_{i}-\frac{1}{k} \sum_{j=1}^{k} \Lambda_{j}\right)^{2}
 +\frac{1}{k^{2}}\frac{4}{n} \sum_{i,j=1}^{k} \left(  \Lambda_{i}^{\frac{1}{2}}\Lambda_{j} -\Lambda_{i}^{\frac{3}{2}}\right)\Bigg\}^{\frac{1}{2}}.\end{aligned}\end{equation}
 From \eqref{inq-4.11} and \eqref{inq-4.20}, we get \eqref{gap-product-1}.  By modifying Cheng and Yang's proof presented in \cite{CY4}, one can get \eqref{gap-product-2}.
Thus, this completes the proof of Corollary \ref{cor5.10}.

\end{proof}

\section{Further Applications}\label{sec5}
In this section, we would like to give some applications of Theorem \ref{thm1.1} and Theorem \ref{thm1.2}. Specially, we obtain some eigenvalue inequalities on the submanifolds of the
Euclidean spaces, unit spheres and projective spaces.
\subsection{Eigenvalues on the Manifolds of Euclidean Space and Unit Sphere}\label{subsec5.1}
Firstly, we suppose that $\mathcal{M}^{n}$ is an $n$-dimensional complete submanifold isometrically embedded into the $(n+p)$-dimensional Euclidean space $\mathbb{R}^{n+p}$ with the mean curvature $H\equiv0$. Then, according to Theorem \ref{thm1.1},  one can deduce the following corollary.
\begin{corr}\label{corr-6.1}Let $(\mathcal{M}^{n},g)$ be an $n$-dimensional complete minimal submanifold isometrically embedded into the Euclidean space $\mathbb{R}^{n+p}$. Then, for any $j=1,2,\cdots$,
eigenvalues of the Dirichlet problem \eqref{diri-prob} of  the Xin-Laplacian satisfy
\begin{equation*}
\begin{aligned}
\sum^{k}_{i=1}(\Lambda_{k+1}-\Lambda_{i})^{2}
&\leq\frac{4}{n}\sum^{k}_{i=1}(\Lambda_{k+1}-\Lambda_{i})
\left(\Lambda_{i}+D_{4}\Lambda_{i}^{\frac{1}{2}}+ \frac{1}{4}D_{4}^{2}\right),
\end{aligned}
\end{equation*}and
\begin{equation*}
\begin{aligned}
\sum_{l=1}^{n}(\Lambda_{j+l}-\Lambda_{j})\leq
4\Lambda_{j}+4 D_{4}\Lambda_{j}^{\frac{1}{2}}+ D_{4}^{2},
\end{aligned}
\end{equation*}
where $D_{4}$ is given by
$
D_{4}=\max_{\Omega}|\nu^{\top}|_{g_{0}}.
$  \end{corr}

Next, we consider that $(\mathcal{M}^{n},g)$ is an $n$-dimensional submanifold isometrically immersed in the unit sphere $\mathbb{S}^{n+p-1}(1) \subset \mathbb{R}^{n+p}$  with mean curvature vector $ \overline{\textbf{H}}$. We use $\overline{\Psi}$ to denote the set of all isometric immersions from $\mathcal{M}^{n}$ into  the unit sphere $\mathbb{S}^{n+p-1}(1)$. By Theorem \ref{thm1.1}, we have the following corollary.
\begin{corr}\label{corr-6.2}
If $(\mathcal{M}^{n},g)$ be an $n$-dimensional submanifold isometrically immersed in the unit sphere $\mathbb{S}^{n+p-1}(1) \subset \mathbb{R}^{n+p}$  with mean curvature vector $ \overline{\emph{\textbf{H}}}$. Then,
eigenvalues of the Dirichlet problem \eqref{diri-prob} of  the Xin-Laplacian satisfy

\begin{equation}\label{Sub-Sph}
\begin{aligned}
\sum^{k}_{i=1}(\Lambda_{k+1}-\Lambda_{i})^{2}
&\leq\frac{4}{n}\sum^{k}_{i=1}(\Lambda_{k+1}-\Lambda_{i})
\left[\Lambda_{i}+D_{5}\Lambda_{i}^{\frac{1}{2}}+ \frac{1}{4}\left(D_{5}^{2}+C_{5}\right)\right],
\end{aligned}
\end{equation}and, for any $j=1,2,\cdots$,
\begin{equation}
\begin{aligned}\label{Sub-Sph-s}
\sum_{l=1}^{n}(\Lambda_{j+l}-\Lambda_{j})\leq
4\Lambda_{j}+4 D_{5}\Lambda_{j}^{\frac{1}{2}}+ D_{5}^{2}+C_{5},
\end{aligned}
\end{equation}
where $C_{5}=\inf_{\overline{\sigma}\in\overline{\Psi}}\max_{\Omega}n^{2}(|\overline{\emph{\textbf{H}}}|^{2}+1)\ \ and \ \ D_{5}=\max_{\Omega}|\nu^{\top}|_{g_{0}}.$

\end{corr}

\begin{proof}
Since the unit sphere can be canonically imbedded into Euclidean space, we have
the following diagram:

\begin{equation*}\begin{aligned}
\xymatrix{
  \mathcal{M}^{n}\ar[dr]_{\psi\circ \sigma} \ar[r]^{\sigma}
                & \mathbb{S}^{n+p-1} \ar[d]^{\psi}  \\
                & \mathbb{R}^{n+p}              }\end{aligned} \end{equation*}
where $\psi: \mathbb{S}^{n+p-1}(1)\rightarrow \mathbb{R}^{n+p}$ is the canonical imbedding from the unit sphere $S^{n+p-1}(1)$ into $\mathbb{R}^{n+p},$ and  $\sigma: \mathcal{M}^{n}\rightarrow \mathbb{S} ^{n+p-1}(1)$ is an isometrical immersion. Then, the composite map $\psi \circ \sigma: \mathcal{M}^{n} \rightarrow \mathbb{R}^{n+p}$ is an isometric immersion from $\mathcal{M}^{n}$ to $\mathbb{R}^{n+p} .$ Let $\overline{\textbf{H}}$ and $\textbf{H}$ be the mean curvature vector fields of $\sigma$ and $\psi \circ \sigma,$ respectively. Then, we have
$
\left| \textbf{H}\right|^{2}=|\overline{\textbf{H}}|^{2}+1.
$
Applying Theorem \ref{thm1.1} directly, we can get \eqref{Sub-Sph} and \eqref{Sub-Sph-s}. Therefore, we finish the proof of this corollary.\end{proof}

In particular, we assume that $(\mathcal{M}^{n},g)$ is an $n$-dimensional unit sphere $\mathbb{S}^{n}(1)$, and then, the mean curvature equals to $1$. This is, $\left| \overline{\textbf{H}}\right|=0$, and thus, we have $\left| \textbf{H}\right|=1$. Therefore, by Corollary \ref{corr-6.2}, we obtain the following corollary.

\begin{corr}\label{corr-6.3}Let $(\mathcal{M}^{n},g)$ be an $n$-dimensional unit sphere $\mathbb{S}^{n}(1)$ and $\Omega$ is a bounded domain on  $\mathbb{S}^{n}(1)$. Then,
eigenvalues of the Dirichlet problem \eqref{diri-prob} of  the Xin-Laplacian satisfy

\begin{equation*}
\begin{aligned}
\sum^{k}_{i=1}(\Lambda_{k+1}-\Lambda_{i})^{2}
&\leq\frac{4}{n}\sum^{k}_{i=1}(\Lambda_{k+1}-\Lambda_{i})
\left[\Lambda_{i}+D_{5}\Lambda_{i}^{\frac{1}{2}}+ \frac{1}{4}\left(D_{5}^{2}+n^{2}\right)\right],
\end{aligned}
\end{equation*}and, for any $j=1,2,\cdots$,
\begin{equation*}
\begin{aligned}
\sum_{l=1}^{n}(\Lambda_{j+l}-\Lambda_{j})\leq
4\Lambda_{j}+4 D_{5}\Lambda_{j}^{\frac{1}{2}}+ D_{5}
^{2}+n^{2},
\end{aligned}
\end{equation*}
where $D_{5}$ is given by $D_{5}=\frac{1}{4}\max_{\Omega}|\nu^{\top}|_{g_{0}}.$
\end{corr}
\subsection{Eigenvalues on the Submanifolds of the Projective Spaces}\label{subsec5.2}
Next, let us recall some results for submanifolds on the projective spaces. For more details, we refer the readers to \cite{Chb,SHI}. Let $\mathbb{F}$ denote the field $\mathbb{R}$ of real numbers,
the field $\mathbb{C}$ of complex numbers or the field $\mathbb{Q}$ of quaternions.
  In a natural way, $\mathbb{R} \subset \mathbb{C} \subset \mathbb{Q}$. For each
element $z$ of $\mathbb{F}$, we define the conjugate of $z$ as follows:
If
$
z=z_{0}+z_{1} i+z_{2} j+z_{3} k \in \mathbb{Q}
$
with $z_{0}, z_{1}, z_{2}, z_{3} \in \mathbb{R} ,$ then
$
\bar{z}=z_{0}-z_{1} i-z_{3} j-z_{3} k.
$
If $z$ is in $\mathbb{C}, \bar{z}$ coincides with the ordinary complex conjugate.
Let us denote by $\mathbb{F}P^{m}$ the $m$-dimensional real projective space if $\mathbb{F}= \mathbb{R}$, the complex projective space with real dimension $2 m$ if $\mathbb{F}= \mathbb{C}$, and the quaternionic projective space with real dimension $4 m$ if $\mathbb{F}= \mathbb{Q}$, respectively. For convenience, we
introduce the integers

\begin{equation}\label{df}
d_{\mathbb{F}}=\operatorname{dim}_{\mathbb{R}}\mathbb{F}=\left\{\begin{array}{ll}
1, & \text { if } \mathbb{F} = \mathbb{R}; \\
2, & \text { if } \mathbb{F} = \mathbb{C}; \\
4, & \text { if } \mathbb{F} = \mathbb{Q}.
\end{array}\right.
\end{equation} It is well known that the manifold $\mathbb{F}P^{m}$ carries a canonical metric so that the Hopf fibration
$\rho: \mathbb{S}^{d_{\mathbb{F}} \cdot(m+1)-1} \subset \mathbb{F}^{m+1} \rightarrow \mathbb{F}P^{m}$ is a Riemannian submersion.
Hence, the sectional curvature of $\mathbb{R}P^{m}$ is $1$, the holomorphic sectional curvature is $4$ and the quaternion sectional  curvature is $4$. We use $\mathcal{A}$ to denote the space of all $(m+1)\times(m+1)$ matrices over $\mathbb{F}$ and let
$$\mathcal{H}_{m+1}(\mathbb{F})=\left\{A \in \mathcal{A} _{m+1}(\mathbb{F}) \mid A^{*}:=\overline{^{t} A}=A\right\}$$ be the vector space of $(m+1) \times(m+1)$ Hermitian
matrices with coefficients in the field $\mathbb{F}$. Then, one can endow $\mathcal{H}_{m+1}( \mathbb{F})$ with an inner product as follows:
$
\langle A, B\rangle=\frac{1}{2} \operatorname{tr}(A B),
$
where ${\rm tr}(\cdot)$ denotes the trace for the given $(m+1) \times(m+1)$ matrix. It is clear that the map $$\rho: \mathbb{S} ^{d_{\mathbb{F}} \cdot(m+1)-1} \subset \mathbb{F} ^{m+1} \rightarrow \mathcal{H}_{m+1}(\mathbb{F})$$ given by
\[
\rho(\textbf{z})=\textbf{z}\textbf{z}^{\ast}=\left(\begin{array}{llll}
\left|z_{0}\right|^{2} & z_{0} \overline{z_{1}} & \cdots & z_{0} \overline{z_{m}} \\
z_{1} \overline{z_{0}} & \left|z_{1}\right|^{2} & \cdots & z_{1} \overline{z_{m}} \\
\cdots & \cdots & \cdots & \cdots \\
z_{m} \overline{z_{0}} & z_{m} \overline{z_{1}} & \cdots & \left|z_{m}\right|^{2}
\end{array}\right)
\]
induces an isometric embedding $\rho$ from $\mathbb{F}P^{m}$ into $\mathcal{H}_{m+1}( \mathbb{F})$ through the Hopf fibration, where
$\textbf{z}=(z_{0},z_{1},\cdots,z_{m})\in\mathbb{S} ^{d_{\mathbb{F}} \cdot(m+1)-1}.$ Moreover, $\rho\left( \mathbb{F}P^{m}\right)$ is a minimal submanifold of the hypersphere $\mathbb{S} \left(\frac{I}{m+1}, \sqrt{\frac{m}{2(m+1)}}\right)$ of $\mathcal{H}_{m+1}( \mathbb{F})$ with radius $\sqrt{\frac{m}{2(m+1)}}$ and
center $\frac{I}{m+1}$, where $I$ is the identity matrix.
In addition, we need the following lemma (see Lemma 6.3 in Chapter 4 in \cite{Chb}, \cite{Sa} and a proof of this lemma in \cite{T}):

\begin{lem} \label{lem-proj}Let $f: \mathcal{M}^{n}  \rightarrow \mathbb{F} P^{\text {m }}$ be an isometric immersion, and let $\widehat{\textbf{H}}$ and $\textbf{H}$ be the mean curvature
vector fields of the immersions $f$ and $\rho \circ f,$ respectively (here $\rho$ is the induced isometric embedding $\rho$ from $\mathbb{F}P^{m}$ into $\mathcal{H}_{m+1}( \mathbb{F})$ explained above). Then, we have
\[
\left| \textbf{H}\right|^{2}=|\widehat{\textbf{H}}|^{2}+\frac{4(n+2)}{3 n}+\frac{2}{3 n^{2}} \sum_{i \neq j} K\left(e_{i}, e_{j}\right),
\]
where $\left\{e_{i}\right\}_{i=1}^{n}$ is a local orthonormal basis of $\overline{\Gamma}(T \mathcal{M}^{n})$ and $K$ is the sectional curvature of $\mathbb{F}P^{m}$ expressed $b y$
\[
K\left(e_{i}, e_{j}\right)=\left\{\begin{array}{ll}
1, & \text { if } \mathbb{F} = \mathbb{R}; \\
1+3\left(e_{i} \cdot J e_{j}\right)^{2}, & \text { if } \mathbb{F} = \mathbb{C}; \\
1+\sum_{r=1}^{3} 3\left(e_{i} \cdot J_{r} e_{j}\right)^{2}, & \text { if } \mathbb{F} = \mathbb{Q},
\end{array}\right.
\]
where $J$ is the complex structure of $\mathbb{C}P^{m}$ and $J_{r}$ is the quaternionic structure of $\mathbb{Q}P ^{m}$.\end{lem}

$$\eqno\Box$$

Therefore, one can infer from Lemma \ref{lem-proj} that

\begin{equation*}
\left|\textbf{H}\right|^{2}=\left\{\begin{array}{ll}
|\widehat{\textbf{H}}|^{2}+\frac{2(n+1)}{2 n}, & \text { for } \mathbb{R} P^{m}; \\
|\widehat{\textbf{H}}|^{2}+\frac{2(n+1)}{2 n}+\frac{2}{n^{2}} \sum_{i, j=1}^{n}\left(e_{i} \cdot J e_{j}\right)^{2} \leq|\widehat{\textbf{H}}|^{2}+\frac{2(n+2)}{n}, & \text { for } \mathbb{C} P^{m}; \\
|\widehat{\textbf{H}}|^{2}+\frac{2(n+1)}{2 n}+\frac{2}{n^{2}} \sum_{i, j=1}^{n} \sum_{r=1}^{3}\left(e_{i} \cdot J_{r} e_{j}\right)^{2} \leq|\widehat{\textbf{H}}|^{2}+\frac{2(n+4)}{n}, & \text { for } \mathbb{Q} P^{m}.
\end{array}\right.
\end{equation*}Hence,  from the above equation, one can verify the following inequality:

\begin{equation}\label{HH}
\left| \textbf{H}\right|^{2} \leq|\widehat{\textbf{H}}|^{2}+\frac{2\left(n+d_{\mathbb{F}}\right)}{n}.
\end{equation}
We note that the equality in \eqref{HH} holds if and only if $\mathcal{M}^{n}$ is a complex submanifold of $\mathbb{C}P^{m}$ (for the case $\mathbb{C}P^{m}$ ) while $n \equiv 0(\bmod 4)$ and $\mathcal{M}^{n}$ is an invariant submanifold of $\mathbb{Q}P^{m}\left(\text { for the case } \mathbb{Q}P^{m}\right)$. By Theorem \ref{thm1.1} and \ref{thm1.2}, we can show the following corollary.

\begin{corr}\label{corr-6.5}
If $\mathcal{M}^{n}$ is isometrically immersed in a projective space $\mathbb{F}P^{m}$ with mean curvature vector $\widehat{\textbf{H}}$, Then,
eigenvalues of the Dirichlet problem \eqref{diri-prob} of  the Xin-Laplacian satisfy

\begin{equation}\label{proj-inequa}
\begin{aligned}
\sum^{k}_{i=1}(\Lambda_{k+1}-\Lambda_{i})^{2}
&\leq\frac{4}{n}\sum^{k}_{i=1}(\Lambda_{k+1}-\Lambda_{i})
\left[\Lambda_{i}+D_{6}\Lambda_{i}^{\frac{1}{2}}+ \frac{1}{4}\left(D_{6}^{2}+C_{6}\right)\right],
\end{aligned}
\end{equation}and, for any $j=1,2,\cdots$,
\begin{equation}
\begin{aligned}\label{proj-inequa-1}
\sum_{l=1}^{n}(\Lambda_{j+l}-\Lambda_{j})\leq
4\Lambda_{j}+4 D_{6}\Lambda_{j}^{\frac{1}{2}}+ D_{6}^{2}+C_{6},
\end{aligned}
\end{equation}
where
$
C_{6}=\frac{1}{4}\inf _{\psi \in \Psi}\max_{\Omega}\left(n^{2}H^{2}+2n\left(n+d_{\mathbb{F} }\right)\right),
$ and $D_{6}=\frac{1}{4}\max_{\Omega} |\nu^{\top}|_{g_{0}},$ and
$d_{\mathbb{F}}=\operatorname{dim}_{ \mathbb{R} } \mathbb{F}$ defined by \eqref{df}.

\end{corr}

\begin{proof}
As we know, there exists a canonical imbedding map $\rho: \mathbb{F}P^{m} \rightarrow \mathcal{H}_{m+1}( \mathbb{F})$ from $\mathbb{F} P^{m}( \mathbb{F} = \mathbb{R} , \mathbb{C} , \mathbb{Q} )$ to Euclidean space $\mathcal{H}_{m+1}( \mathbb{F} )$. Therefore, for compact manifold $\mathcal{M}^{n}$ isometrically immersed into the projective space $\mathbb{F} P^{m},$ one has the following diagram:

\begin{equation*}\begin{aligned}
\xymatrix{
  \mathcal{M}^{n}\ar[dr]_{\rho\circ f} \ar[r]^{f}
                & \mathbb{F}P^{m} \ar[d]^{\rho}  \\
                &\mathcal{H}_{m+1}(\mathbb{F})             }\end{aligned} \end{equation*}
 $f: \mathcal{M}^{n} \rightarrow$
$\mathbb{F}P^{m}$ denotes  an isometric immersion from $\mathcal{M}^{n}$ to $\mathbb{F}P^{m}$. Then, the composite map $\rho \circ f: \mathcal{M}^{n} \rightarrow \mathcal{H}_{m+1}( \mathbb{F})$ is an isometric immersion from $\mathcal{M}^{n}$ to $\mathcal{H}_{m+1}( \mathbb{F})$. According to inequality \eqref{HH} and Theorem \ref{thm1.1}, we can conclude \eqref{proj-inequa}  and \eqref{proj-inequa-1}. Hence, it completes the proof of corollary \ref{corr-6.5}.

\end{proof}

\subsection{Manifolds Admitting Some Special Functions}\label{subsec5.3}
In this section, we would like to discuss the eigenvalue of Xin-Laplacian on the manifolds admitting some special functions.
\begin{thm}\label{thm4.1-1}
Let $\mathcal{M}^{n}$ be an $n$-dimensional complete Riemannian manifold and let $\Omega$ be a bounded domain with smooth boundary in $\mathcal{M}^{n}$. Denote by $\Lambda_{i}$ the $i$ -th eigenvalue of the problem \eqref{diri-prob} of the differential operator $\mathfrak{L}_{\nu}$. If there exist a function $\mathcal{W}: \Omega \rightarrow \mathbb{R}$ and a positive constant $C_{3}$ such that $
|\nabla \mathcal{W}|_{g}=1$, and $|\Delta \mathcal{W}|_{a}\leq C_{3}$, where $|w|_{a}$ denotes the absolute value of $w$,
then
\begin{equation}
\begin{aligned}\label{eq-5.7sumD3}
\sum_{i=1}^{k}\left(\Lambda_{k+1}-\Lambda_{i}\right)^{2} \leq\sum_{i=1}^{k}\left(\Lambda_{k+1}-\Lambda_{i}\right)\left(4 \Lambda_{i}+4\left(C_{3}+D_{3}\right) \Lambda_{i}^{\frac{1}{2}}+\left(C_{3}+D_{3}\right)^{2}\right),
\end{aligned}
\end{equation}where $D_{3}=\max_{\Omega}|\nu^{\top}|_{g_{0}}$.

\end{thm}

\begin{rem} Let $\mathcal{M}^{n}$ be an $n$-dimensional connected complete Riemannian manifold whose Ricci curvature satisfies $\operatorname{Ric}_{\mathcal{M}^{n}} \geq-(n-1) \kappa^{2}, \kappa \geq 0 .$ Suppose that there is a smooth function $\mathcal{W}$ on $\mathcal{M}^{n}$ with $|\nabla \mathcal{W}|_{g}=1$. Then, we have $|\Delta \mathcal{W}|_{a} \leq(n-1) \kappa^{2} .$ See Remark 3.6 in {\rm \cite{S}}.
Furthermore, we consider that $\xi:[0,+\infty) \rightarrow M$ is a geodesic ray, namely a unit speed geodesic with $d(\xi(s), \xi(t))=t-s$ for any $t>s>0 .$ Then the Busemann function $b_{\xi}$ corresponding to $\xi$ is defined as $b_{\xi}(q):=\lim _{t \rightarrow+\infty}(d(q, \xi(t))-t).$
If $\mathcal{M}^{n}$ is an Hadamard manifold, then it is known that $b_{\xi}$ is a convex function of class $C^{2}$ with $|\nabla b_{\xi}|_{g}\equiv 1$ and these conditions characterize Busemann functions (see  {\rm \cite{BGS,HH}}).  Thus, the Bussemann functions on the Cartan-Hadamard manifolds $\mathcal{M}^{n}$ satisfy the conditions in  Theorem {\rm\ref{thm4.1-1}}. Also, if $\mathcal{N}^{n-1}$ is complete Riemannian manifold with Ricci curvature bounded below and if $\mathcal{M}^{n}=\mathcal{N}^{n-1} \times \mathbb{R}$ is the product of $\mathcal{N}$ and $\mathbb{R}$ with the product metric, then the function $f: \mathcal{M}^{n} \rightarrow \mathbb{R}$ given by $f(p, t)=t$ satisfies the conditions  of Theorem  {\rm\ref{thm4.1-1}}.  \end{rem}

\begin{rem}  Let $\mathcal{M}^{n}= \mathbb{R} \times \mathcal{N}^{n-1}$ be the complete manifold with the warped product metric $d s_{M}^{2}=$ $d t^{2}+\exp (2 t) d s_{N}^{2},$ where $\mathcal{N}^{n-1}$ is a complete manifold. If the Ricci curvature of $\mathcal{N}$ is non-negative, then $\operatorname{Ric}_{\mathcal{M}^{n}} \geq-(n-1)$, which means that the function $f: \mathcal{M}^{n} \rightarrow \mathbb{R}$ given by $f(p, t)=t$ satisfies $|\nabla f|_{g}=1 \ \  and \ \ |\Delta f|_{a} \leq n-1.$ See {\rm \cite{Sa}} for details. Consequently,
the product Riemannian manifold $\mathcal{M}^{n}$ satisfies the condition in Theorem {\rm \ref{thm4.1-1}}.
\end{rem}

\vskip3mm

\noindent \emph{Proof of Theorem} \ref{thm4.1-1}.
Substituting $\varphi=\mathcal{W}$ into \eqref{2.11}, and utilizing \eqref{new-lemma1} and Cauchy-Schwarz inequality, we infer that

\begin{equation*}
\begin{aligned}&\sum_{i=1}^{k}\left(\Lambda_{k+1}-\Lambda_{i}\right)^{2}  \\&\leq \sum_{i=1}^{k}\left(\Lambda_{k+1}-\Lambda_{i}\right) \int_{\Omega}\left(u_{i}\left(\Delta \mathcal{W}+\langle\nu, \nabla \mathcal{W}\rangle_{g_{0}}\right)+2\left\langle\nabla \mathcal{W}, \nabla u_{i}\right\rangle_{g}\right)^{2}e^{\langle\nu,X\rangle_{g_{0}}}dv \\&
\leq \sum_{i=1}^{k}\left(\Lambda_{k+1}-\Lambda_{i}\right) \int_{\Omega}\left(\left|u_{i}\right|(|\Delta \mathcal{W}|_{a}+|\nu^{\top}|_{g_{0}}|\nabla \mathcal{W}|_{g}) +2|\nabla \mathcal{W}|_{g} \left|\nabla u_{i}\right|_{g}\right)^{2} e^{\langle\nu,X\rangle_{g_{0}}}dv,\end{aligned}
\end{equation*}since $|\nabla\mathcal{W}|_{g}=1$. Furthermore,  we have

\begin{equation}\label{rig-in}
\begin{aligned}&\sum_{i=1}^{k}\left(\Lambda_{k+1}-\Lambda_{i}\right)^{2}  \\&\leq\sum_{i=1}^{k}\left(\Lambda_{k+1}-\Lambda_{i}\right) \int_{\Omega}\left[\left(C_{3}+|\nu^{\top}|_{g_{0}}\right)\left|u_{i}\right|_{a}+2\left|\nabla u_{i}\right|_{g}\right]^{2} e^{\langle\nu,X\rangle_{g_{0}}}dv\\
&\leq \sum_{i=1}^{k}\left(\Lambda_{k+1}-\Lambda_{i}\right) \Bigg{[}\left(C_{3}+D_{3}\right)^{2}\left\|u_{i}\right\|_{\Omega}^{2} +4\left(C_{3}+D_{3}\right) \left\| u_{i}\left|\nabla u_{i}\right|_{g} \right\|_{\Omega} +4\left\|\left|\nabla u_{i} \right|_{g}\right\|_{\Omega}^{2}\Bigg{]},\end{aligned}
\end{equation} since $
|\nabla \mathcal{W}|_{g}=1$, and $|\Delta \mathcal{W}|_{a}\leq C_{3}$,  where $D_{3}=|\nu^{\top}|_{g_{0}}.$ By Cauchy-Schwarz inequality, we derive

\begin{equation}\label{cau-schine}
\left\| u_{i}\left|\nabla u_{i}\right|_{g} \right\|_{\Omega}\leq\left(\left\| u_{i} \right\|_{\Omega}\right)^{\frac{1}{2}}\left(\left\|\left|\nabla u_{i}\right|_{g} \right\|_{\Omega}\right)^{\frac{1}{2}}=\Lambda_{i}^{\frac{1}{2}}.
\end{equation}From \eqref{rig-in} and \eqref{cau-schine}, we yield \eqref{eq-5.7sumD3},
as we desired.
Hence, this completes the proof of this theorem.
$$\eqno\Box$$

\begin{thm}\label{thm4.1-2}
Let $\mathcal{M}^{n}$ be an $n$-dimensional complete Riemannian manifold and let $\Omega$ be a bounded domain with smooth boundary in $\mathcal{M}^{n}$. Denote by $\Lambda_{i}$ the $i$ -th eigenvalue of the problem \eqref{diri-prob} of the differential operator $\mathfrak{L}_{\nu}$. If $\Omega$ admits an eigenmap $$f=\left(f_{1}, f_{2}, \cdots, f_{m+1}\right): \Omega \rightarrow \mathbb{S}^{m}(1)$$ corresponding to an eigenvalue $\eta,$ that is,
$
\Delta f_{\alpha}=-\eta f_{\alpha},\ \ where \ \ \alpha=1, \cdots, m+1,$  and
$ \sum_{\alpha=1}^{m+1} f_{\alpha}^{2}=1,
$
then
\begin{equation}
\begin{aligned}\label{inequlity-4.2}
\sum_{i=1}^{k}\left(\Lambda_{k+1}-\Lambda_{i}\right)^{2} \leq \sum_{i=1}^{k}\left(\Lambda_{k+1}-\Lambda_{i}\right)\left( 4 \Lambda_{i}+4 D_{3}\Lambda_{i}^{1 / 2}+D_{3}^{2}+ \eta\right),
\end{aligned}
\end{equation}
where $\mathbb{S}^{m}(1)$ is the unit sphere of dimension $m$ and $D_{3}=\max_{\Omega}|\nu^{\top}|_{g_{0}}.$

\end{thm}

\begin{rem}  Let $\mathcal{M}^{n}$ be a compact homogeneous Riemannian manifold. Then, Riemannian manifold $\mathcal{M}^{n}$ admits eigenmaps to some unit sphere for the first positive eigenvalue of the Laplacian. See {\rm \cite{Li1}}. Therefore, it satisfies the condition of Theorem {\rm\ref{thm4.1-2}}.
\end{rem}

\vskip 3mm
\noindent \emph{Proof of Theorem} \ref{thm4.1-2}.
Taking the Laplacian for the following equation

\begin{equation}\label{sum-4.3}
\sum_{\alpha=1}^{m+1} f_{\alpha}^{2}=1
\end{equation}
and using the fact that
$\Delta f_{\alpha}=-\eta f_{\alpha}, \ \ {\rm where}  \ \ \alpha=1, \cdots, m+1,$ we have

\begin{equation}\label{m+1-ineq1}
\sum_{\alpha=1}^{m+1}\left|\nabla f_{\alpha}\right|_{g}^{2}=\eta.
\end{equation}
Taking the gradient for the equation \eqref{sum-4.3}, we have

\begin{equation}\label{m+1-ineq2}
\sum_{\alpha=1}^{m+1} f_{\alpha} \nabla f_{\alpha}=\textbf{0}.
\end{equation}
By Cauchy-Schwarz inequality, \eqref{ineq-nu-w} and \eqref{m+1-ineq1}, we have
\begin{equation}\begin{aligned}\label{18-ineq}
 \int_{\Omega} \sum_{\alpha=1}^{m+1}\left(u_{i}\left\langle\nu, \nabla f_{\alpha}\right\rangle_{g_{0}}\right)^{2} e^{\langle\nu,X\rangle_{g_{0}}}dv \leq\eta\int_{\Omega} u_{i}^{2}|\nu^{\top}|^{2}_{g_{0}} e^{\langle\nu,X\rangle_{g_{0}}}dv,
\end{aligned}\end{equation}

\begin{equation}\begin{aligned}\label{19-ineq}
 \int_{\Omega} \sum_{\alpha=1}^{m+1}\left(4u_{i}\left\langle\nu, \nabla f_{\alpha}\right\rangle_{g_{0}}\left\langle \nabla u_{i}, \nabla f_{\alpha}\right\rangle_{g}\right) e^{\langle\nu,X\rangle_{g_{0}}}dv\leq\eta\int_{\Omega}  \left(4u_{i}|\nu^{\top}|_{g_{0}}| \nabla u_{i}|_{g}\right) e^{\langle\nu,X\rangle_{g_{0}}}dv,
\end{aligned}\end{equation}
and

\begin{equation}\begin{aligned}\label{ineq-20}
 4\int_{\Omega} \sum_{\alpha=1}^{m+1}\left\langle \nabla u_{i}, \nabla f_{\alpha}\right\rangle_{g}^{2} e^{\langle\nu,X\rangle_{g_{0}}}dv\leq4\eta\int_{\Omega}|\nabla u_{i}|^{2}_{g}e^{\langle\nu,X\rangle_{g_{0}}}dv.
\end{aligned}\end{equation}From \eqref{m+1-ineq1},  we have

\begin{equation}\label{left-ineq}\sum_{\alpha=1}^{m+1}\sum^{k}_{i=1}(\Lambda_{k+1}-\Lambda_{i})^{2}\|u_{i}\nabla
f_{\alpha}\|_{\Omega}^{2}=\eta  \sum_{i=1}^{k}\left(\Lambda_{k+1}-\Lambda_{i}\right)^{2}.\end{equation}Taking $\varphi=f_{\alpha}$ in \eqref{2.11} and summing over $\alpha$, we infer that
\begin{equation}
\begin{aligned}\label{sum-ineq-1a}
\sum_{\alpha=1}^{m+1}\sum^{k}_{i=1}(\Lambda_{k+1}-\Lambda_{i})^{2}\|u_{i}\nabla
f_{\alpha}\|_{\Omega}^{2} &\leq\sum_{\alpha=1}^{m+1}\sum^{k}_{i=1}(\Lambda_{k+1}-\Lambda_{i})
\|2\langle\nabla f_{\alpha},\nabla u_{i}\rangle_{g}+u_{i}\mathfrak{L}_{\nu}f_{\alpha}\|_{\Omega}^{2}.
\end{aligned}
\end{equation}
Substituting \eqref{ineq-nu-w}, \eqref{m+1-ineq1}-\eqref{left-ineq} into \eqref{left-ineq}, we conclude that,

\begin{equation*}\begin{aligned}
\eta & \sum_{i=1}^{k}\left(\Lambda_{k+1}-\Lambda_{i}\right)^{2} \\
& \leq \sum_{i=1}^{k}\left(\Lambda_{k+1}-\Lambda_{i}\right) \int \sum_{\Omega}^{m+1}\left(u_{i}\left(\Delta f_{\alpha}+\left\langle\nu, \nabla f_{\alpha}\right\rangle_{g_{0}}\right)+2\left\langle\nabla f_{\alpha}, \nabla u_{i}\right\rangle_{g}\right)^{2} e^{\langle\nu,X\rangle_{g_{0}}}dv\\
&=\sum_{i=1}^{k}\left(\Lambda_{k+1}-\Lambda_{i}\right) \int_{\Omega}\sum_{\alpha=1}^{m+1}\left(-\eta u_{i} f_{\alpha}+u_{i}\left\langle\nu, \nabla f_{\alpha}\right\rangle_{g_{0}}+2\left\langle \nabla u_{i}, \nabla f_{\alpha}\right\rangle_{g}\right)^{2} e^{\langle\nu,X\rangle_{g_{0}}}dv.\\
&\leq \sum_{i=1}^{k}\left(\Lambda_{k+1}-\Lambda_{i}\right)\left(\eta^{2}+\int_{\Omega}\left(4\left|\nabla u_{i}\right|_{g}^{2}+4\left|u_{i}\right|_{g}\left|\nabla u_{i}\right|_{g}|\nu^{\top}|_{g_{0}}+u_{i}^{2}|\nu^{\top}|_{g_{0}}^{2}\right) \eta e^{\langle\nu,X\rangle_{g_{0}}}dv\right) \\&
\leq \sum_{i=1}^{k}\left(\Lambda_{k+1}-\Lambda_{i}\right)\left(\eta^{2}+\left(4 \Lambda_{i}+4D_{3}\Lambda_{i}^{1 / 2}+D_{3}^{2}\right) \eta\right),
\end{aligned}\end{equation*}where
$D_{3}=\max_{\Omega}|\nu^{\top}|.$
Thus, we can obtain \eqref{inequlity-4.2}.
$$\eqno\Box$$

\vskip5mm

\section{The Closed Eigenvalue Problem}\label{sec6}
In this section, we investigate eigenvalue inequalities for the closed eigenvalue problem on the compact Riemannian manifolds.

\subsection{ Estimates for the Eigenvalue of Closed Eigenvalue Problem}\label{subsec6.1}

Let $\mathfrak{L}_{\nu}$ be an $n$-dimensional compact Riemannian manifolds without boundary. In this subsection, we would like to study closed eigenvalue problem \eqref{closed-prob} and establish some eigenvalue inequalities.
By the same method as the proof of Proposition
\ref{prop2.2}, one can prove the following propsition.
\begin{prop}
\label{prop7.1} Let $\overline{\phi}_l$, $l=1, 2, \cdots, m$,  be  smooth
functions on an $n$-dimensional closed  Riemannian manifold $\mathcal{M}^{n}$.
Assume that $\overline{\Lambda}_{i}$  is  the $i^{\text{th}}$ eigenvalue of the
closed eigenvalue problem \eqref{closed-prob} and $\overline{u}_{i}$ is an
orthonormal eigenfunction corresponding to $\overline{\Lambda}_{i}$, where $i =0,
1,2,\cdots$, such that
$
\mathfrak{L}_{\nu}\overline{u}_{i} =-\overline{\Lambda}_{i}\overline{u}_{i},$ and
$\int_{\mathcal{M}^{n}}
\overline{u}_{i}\overline{u}_{j}e^{\langle\nu,X\rangle_{g_{0}}}dv=\delta_{ij},
$ for any $i,j=0,1,2,\cdots$. Then, for any $j=0,1, 2, \cdots$,  there exists an
orthogonal matrix $A=(a_{lt})_{m\times m}$ such that $\overline{\Phi}_l=\sum_{s=1}^ma_{ls}\overline{\phi}_s$
satisfy
\begin{equation}\label{2.7}
\sum^{m}_{l=1}(\overline{\Lambda}_{j+l}-\overline{\Lambda}_{j})\|\overline{u}_j\nabla
\overline{\Phi}_{l}\|^{2}_{\mathcal{M}^{n}} \leq
\sum^{m}_{l=1}\int_{\mathcal{M}^{n}}\big(\overline{u}_{j}\mathfrak{L}_{\nu}\overline{\Phi}_{l} +2\langle\nabla
\overline{\Phi}_{l},\nabla \overline{u}_{j}\rangle_{g}\big)^{2}e^{\langle\nu,X\rangle_{g_{0}}}dv.
\end{equation}

\end{prop}

Synthesizing Proposition \ref{prop7.1}, Lemma \ref{lem2.5}, Lemma \ref{lem3.3} and Lemma \ref{lem3.2}, we can prove the following theorem.

\begin{thm}
\label{thm7.1}
Let $\mathcal{M}^{n}$ be an $n$-dimensional compact Riemannian
manifold without boundary isometrically embedded into the Euclidean space $\mathbb{R}^{n+p}$.  Then, for any $j=0,1,2,\cdots$, the eigenvalues of closed eigenvalue  problem \eqref{closed-prob}
of Xin-Laplacian satisfy

\begin{equation}
\begin{aligned}
\label{z-ineq-2}\sum^{n}_{k=1}\overline{\Lambda}_{j+k} \leq(n+4)\overline{\Lambda}_{j} +4\overline{D}_{1}\overline{\Lambda}_{j}^{\frac{1}{2}}+\overline{D}^{2}_{1}+\overline{C}_{1},
\end{aligned}
\end{equation}and
\begin{equation}
\begin{aligned}
\label{z-ineq-2-1} \sum^{n}_{k=1}\overline{\Lambda}_{j+k} \leq(n+6)\overline{\Lambda}_{j}+3\overline{D}^{2}_{1}+\overline{C}_{1},
\end{aligned}
\end{equation}
where
$\overline{C}_{1}= \inf_{\psi\in \Psi}\max_{\mathcal{M}^{n}}n^{2}H^{2} \ \ and \ \ \overline{D}_{1}= \max_{\mathcal{M}^{n}}|\nu^{\top}|_{g_{0}}.$
\end{thm}
\begin{proof}By making use of the same proof as in the proof of Theorem \ref{thm1.2}, we can
prove this theorem if one notices to count the number of eigenvalues from $1$. \end{proof}

For the sake of  the appearance
of the mean curvature, it is very natural to generalize an important result obtained by Rielly in \cite{R}.
Indeed, by Theorem \ref{thm7.1}, we have following corollary.
\begin{corr}
\label{c7.2}
Let $\mathcal{M}^{n}$ be an $n$-dimensional compact Riemannian
manifold without boundary.   Then, for any $j=0,1,2,\cdots$, the eigenvalues of closed eigenvalue  problem \eqref{closed-prob}
of Xin-Laplacian satisfy
\begin{equation}
\begin{aligned}
\label{z-ineq-5} \sum^{n}_{k=1}\overline{\Lambda}_{j+k} \leq(n+4)\overline{\Lambda}_{j} +\int_{\mathcal{M}^{n}} \overline{u}_{j}^{2}\left(n^{2}H^{2} +|\nu^{\top}|_{g_{0}}^{2}+4\overline{\Lambda}_{j}^{\frac{1}{2}}|\nu^{\top}|_{g_{0}}\right)e^{\langle\nu,X\rangle_{g_{0}}}dv,
\end{aligned}
\end{equation}and
\begin{equation}
\begin{aligned}
\label{z-ineq-5} \sum^{n}_{k=1}\overline{\Lambda}_{j+k} \leq(n+6)\overline{\Lambda}_{j} +\int_{\mathcal{M}^{n}} \overline{u}_{j}^{2}\left(n^{2}H^{2} +3|\nu^{\top}|_{g_{0}}^{2}\right)e^{\langle\nu,X\rangle_{g_{0}}}dv.
\end{aligned}
\end{equation}

\end{corr}

\begin{corr}
\label{cor7.4} For  an $n$-dimensional complete submanifold $\mathcal{M}^{n}$
in  the Euclidean space $\mathbb{R}^{n+p}$,
 eigenvalues  of the closed eigenvalue  problem \eqref{closed-prob} of the
differential operator $\mathfrak{L}_{\nu}$ satisfy

\begin{equation}
\begin{aligned}\label{Rielly-type-2}
\sum^{n}_{k=1}\overline{\Lambda}_{k} \leq
\dfrac{\int_{\mathcal{M}^{n}}\left(n^{2}H^{2}+3|\nu^{\top}|_{g_{0}}^{2}\right)e^{\langle\nu,X\rangle_{g_{0}}}dv}{\int_{\mathcal{M}^{n}}e^{\langle\nu,X\rangle_{g_{0}}}dv}.
\end{aligned}
\end{equation}
\end{corr}
\begin{proof}
Since $\overline{\Lambda}_0=0$ and $\overline{u}_0$ is constant, by taking $j=0$ in the Theorem \ref{thm7.1}, we can infer that,
\begin{equation*}
\begin{aligned}
 \sum^{n}_{k=1}\overline{\Lambda}_{k} \leq
\dfrac{\int_{\mathcal{M}^{n}}\left(n^{2}H^{2}+3|\nu^{\top}|_{g_{0}}^{2}\right)e^{\langle\nu,X\rangle_{g_{0}}}dv}{\int_{\mathcal{M}^{n}}e^{\langle\nu,X\rangle_{g_{0}}}dv},
\end{aligned}
\end{equation*}
where we have used a fact as follows: $$\int_{\mathcal{M}^{n}}\overline{u}_0^2e^{\langle\nu,X\rangle_{g_{0}}}dv=1.$$ Therefore,  we finish the proof of this corollary.
\end{proof}
\begin{rem}\label{rem5.1}
If we take $\nu=0$, the operator $\mathfrak{L}_{\nu}$ is the Beltrami-Laplacian and  we have
\begin{equation*}
\begin{aligned}
\sum^{n}_{k=1}\overline{\Lambda}_{k} \leq
\dfrac{n^{2}\int_{\mathcal{M}^{n}}{H}^{2}dv}{\int_{\mathcal{M}^{n}}dv},
\end{aligned}
\end{equation*}which is a remarkable result obtained by Ilias and Makhoul in {\rm \cite{IM}}. Also, see {\rm \cite{SHI}}.
In particular, when $M^{n}$ is an $n$-dimensional unit sphere $\mathbb{S}^{n}(1)$ and $\nu$ is a zero vector, the identity holds. Hence, our result is a generalization of Reilly's result  in {\rm \cite{R}} on the first
eigenvalue
$$
\overline{\Lambda}_1\leq \dfrac{n\int_{\mathcal{M}^{n}}{H}^{2}dv}{\int_{\mathcal{M}^{n}}dv}.
$$
\end{rem}

Next, we consider that $\mathcal{M}^{n}$ is  an $n$-dimensional  compact minimal submanifold in
the unit sphere $\mathbb{S}^{n+p}(1)$. For this case, we have the following result.
\begin{thm}\label{thm7.6}
Let $\mathcal{M}^{n}$ be  an $n$-dimensional  compact minimal submanifold in
the unit sphere $\mathbb{S}^{n+p}(1)$.  Then, for any $j$, where $j=0,1,2,\cdots$,
eigenvalues of the closed eigenvalue problem \eqref{closed-prob} of the
differential operator $\mathfrak{L}_{\nu}$ satisfy

\begin{equation}
\begin{aligned}
\label{z-ineq-22}\sum^{n}_{k=1}\overline{\Lambda}_{j+k} \leq(n+4)\overline{\Lambda}_{j} +4\overline{D}_{2}\overline{\Lambda}_{j}^{\frac{1}{2}}+\overline{D}^{2}_{2}+n^{2},
\end{aligned}
\end{equation} and

\begin{equation}
\begin{aligned}
\label{z-ineq-2-12} \sum^{n}_{k=1}\overline{\Lambda}_{j+k} \leq(n+6)\overline{\Lambda}_{j}+3\overline{D}_{2}^{2}+n^{2},
\end{aligned}
\end{equation}
where $\overline{D}_{2}= \inf_{\psi\in \Psi}\max_{\mathcal{M}^{n}} |\nu^{\top}|_{g_{0}}.$
\end{thm}

\begin{proof} Since $\mathcal{M}^{n}$ is  an $n$-dimensional minimal
submanifold in the unit sphere $\mathbb{S}^{n+p}(1)$, then $\mathcal{M}^n$ can
be seen as a compact submanifold in $\mathbb{R}^{n+p+1}$ with mean
curvature $H\equiv1$. Therefore, by  Corollary \ref{c7.2}, we know that both inequalities \eqref{z-ineq-22} and \eqref{z-ineq-2-12}
hold.
\end{proof}

By the same strategy as the proof of Proposition
\ref{prop2.3}, one also can prove the following proposition.
\begin{prop}\label{prop7.15}
Let $(\mathcal{M}^{n},g)$ be an $n$-dimensional compact Riemannian manifold without boundary.
Assume that $\overline{\Lambda}_{i}$  is  the $i^{\text{th}}$ eigenvalue of the
closed eigenvalue problem \eqref{closed-prob} and $\overline{u}_{i}$ is an
orthonormal eigenfunction corresponding to $\overline{\Lambda}_{i}$, $i =0,
1,2,\cdots$, such that$ \mathfrak{L}_{\nu}\overline{u}_{i} =-\overline{\Lambda}_{i}\overline{u}_{i},$ and $ \int_{\mathcal{M}^{n}}
\overline{u}_{i}\overline{u}_{j}e^{\langle\nu,X\rangle_{g_{0}}}dv=\delta_{ij},$ for any $i,j=0,1,2,\cdots$.
Then, for any function $\overline{\varphi}(x)\in C^{2}(\mathcal{M}^{n})$ and any positive
integer $k$,  eigenvalues of the close eigenvalue problem \eqref{closed-prob} satisfy
\begin{equation*}
\begin{aligned}
\sum^{k}_{i=0}(\overline{\Lambda}_{k+1}-\overline{\Lambda}_{i})^{2}\|\overline{u}_{i}\nabla
\overline{\varphi}\|_{\mathcal{M}^{n}}^{2} \leq\sum^{k}_{i=0}(\overline{\Lambda}_{k+1}-\overline{\Lambda}_{i})
\|2\langle\nabla \overline{\varphi},\nabla \overline{u}_{i}\rangle_{g}+\overline{u}_{i}\mathfrak{L}_{\nu}\overline{\varphi}\|_{\mathcal{M}^{n}}^{2},
\end{aligned}
\end{equation*}
where $
\|\overline{\varphi}(x)\|_{\mathcal{M}^{n}}^{2} =\int_{\mathcal{M}^{n}}\overline{\varphi}^{2}(x)e^{\langle\nu,X\rangle_{g_{0}}}dv.
$
\end{prop}

By using  Proposition \ref{prop7.15}  and Lemma \ref{lem2.5}, we can establish the following eigenvalue inequality of Yang type.

\begin{thm}\label{thm7.16}
Let $(\mathcal{M}^{n},g)$ be an $n$-dimensional closed Riemannian manifold isometrically embedded into the Euclidean space $\mathbb{R}^{n+p}$.
Assume that $\overline{\Lambda}_{i}$  is the $i^{th}$ eigenvalue of eigenvalue problem
\eqref{closed-prob} of  the Xin-Laplacian. Then, we have

\begin{equation}
\begin{aligned}
\label{thm-6.2-1} \sum^{k}_{i=0}(\overline{\Lambda}_{k+1}-\overline{\Lambda}_{i})^{2}
\leq\frac4n&\sum^{k}_{i=0}(\overline{\Lambda}_{k+1}-\overline{\Lambda}_{i})  \left(\overline{\Lambda}_i+\frac{1}{4}\overline{D}_{1}\overline{\Lambda}_{j}^{\frac{1}{2}}+\overline{D}_{1}^{2}+\frac{1}{4}\overline{C}_{1}\right),
\end{aligned}
\end{equation} and

\begin{equation}
\begin{aligned}
\label{thm-6.2-2}\sum^{k}_{i=0}(\overline{\Lambda}_{k+1}-\overline{\Lambda}_{i})^{2}
\leq\frac6n&\sum^{k}_{i=0}(\overline{\Lambda}_{k+1}-\overline{\Lambda}_{i})\left(\overline{\Lambda}_i+3\overline{D}_{1}^{2}+\frac{1}{6}\overline{C}_{1}\right),
\end{aligned}
\end{equation}
where $\overline{C}_{1}= \inf_{\psi\in \Psi}\max_{\mathcal{M}^{n}}n^{2}H^{2}\ \  and \ \ \overline{D}_{1}= \max_{\mathcal{M}^{n}}|\nu^{\top}|_{g_{0}}.$
\end{thm}
\begin{proof}The proof almost is a copy of the proof of Theorem \ref{thm1.1} word by word, and the only thing needs to be done is to
notices to count the number of eigenvalues from $0$. \end{proof}

By Theorem \ref{thm7.16}, we have following corollary.
\begin{corr}\label{coro7.17}
Let $(\mathcal{M}^{n},g)$ be an $n$-dimensional closed Riemannian manifold isometrically embedded into the Euclidean space $\mathbb{R}^{n+p}$.
Assume that $\overline{\Lambda}_{i}$  is the $i^{th}$ eigenvalue of eigenvalue problem
\eqref{closed-prob} of  the Xin-Laplacian. Then, we have

\begin{equation*}
\begin{aligned}
 \sum^{k}_{i=0}(\overline{\Lambda}_{k+1}-\overline{\Lambda}_{i})^{2}
\leq\frac4n&\sum^{k}_{i=0}(\overline{\Lambda}_{k+1}-\overline{\Lambda}_{i}) \\&\times \left[\overline{\Lambda}_i+\frac{1}{4}\int_{\mathcal{M}^{n}}\overline{u}
_{i}^{2}\left(n^{2}H^{2} +|\nu^{\top}|_{g_{0}}^{2}+4\overline{\Lambda}_{j}^{\frac{1}{2}}|\nu^{\top}|_{g_{0}}\right)e^{\langle\nu,X\rangle_{g_{0}}}dv\right],
\end{aligned}
\end{equation*}
and
\begin{equation*}
\begin{aligned}
 \sum^{k}_{i=0}(\overline{\Lambda}_{k+1}-\overline{\Lambda}_{i})^{2}
\leq\frac6n&\sum^{k}_{i=0}(\overline{\Lambda}_{k+1}-\overline{\Lambda}_{i})   \left[\overline{\Lambda}_i+\frac{1}{6}\int_{\mathcal{M}^{n}}\overline{u}
_{i}^{2}\left(n^{2}H^{2} +3|\nu^{\top}|_{g_{0}}^{2} \right)e^{\langle\nu,X\rangle_{g_{0}}}dv\right].
\end{aligned}
\end{equation*}
\end{corr}

Finally, we assume that $\mathcal{M}^{n}$ is  an $n$-dimensional  compact minimal submanifold in
the unit sphere $\mathbb{S}^{n+p}(1)$. For this case, we have the following theorem.
\begin{thm}\label{thm7.18}
Let $\mathcal{M}^{n}$ be  an $n$-dimensional  compact minimal submanifold in
the unit sphere $\mathbb{S}^{n+p}(1)$.  Then, for any $j$, where $j=0,1,2,\cdots$,
eigenvalues of the closed eigenvalue problem \eqref{closed-prob} of
differential operator $\mathfrak{L}_{\nu}$ satisfy

\begin{equation}
\begin{aligned}
\label{thm-6.2-1} \sum^{k}_{i=0}(\overline{\Lambda}_{k+1}-\overline{\Lambda}_{i})^{2}
\leq\frac4n&\sum^{k}_{i=0}(\overline{\Lambda}_{k+1}-\overline{\Lambda}_{i})  \left(\overline{\Lambda}_i+\frac{1}{4}\overline{D}_{2}\overline{\Lambda}_{j}^{\frac{1}{2}}+\overline{D}_{2}^{2}+\frac{1}{4}n^{2}\right),
\end{aligned}
\end{equation} and

\begin{equation}
\begin{aligned}
\label{thm-6.2-2} \sum^{k}_{i=0}(\overline{\Lambda}_{k+1}-\overline{\Lambda}_{i})^{2}
\leq\frac6n&\sum^{k}_{i=0}(\overline{\Lambda}_{k+1}-\overline{\Lambda}_{i})\left(\overline{\Lambda}_i+3\overline{D}_{2}^{2}+\frac{1}{6}n^{2}\right),
\end{aligned}
\end{equation}
where
$\overline{D}_{2}= \inf_{\psi\in \Psi}\max_{\mathcal{M}^{n}} |\nu^{\top}|_{g_{0}}.$
\end{thm}

\begin{proof} Since $\mathcal{M}^{n}$ is  an $n$-dimensional minimal
submanifold in the unit sphere $\mathbb{S}^{n+p}(1)$, then $\mathcal{M}^n$ can
be viewed as a compact submanifold in $\mathbb{R}^{n+p+1}$ with mean
curvature $H=1$. Therefore, by the Corollary \ref{coro7.17}, it is easy to see that both inequalities \eqref{thm-6.2-1} and \eqref{thm-6.2-2}
hold. This completes the proof of this theorem.
\end{proof}

\begin{rem}In Theorem {\rm \ref{thm7.18}},  we assume that  $\nu=0$, and then, \eqref{thm-6.2-1} implies that

\begin{equation}
\begin{aligned}
\label{yang-closed}
\sum^{k}_{i=0}(\overline{\Lambda}_{k+1}-\overline{\Lambda}_{i})^{2}
\leq\frac4n\sum^{k}_{i=0}(\overline{\Lambda}_{k+1}-\overline{\Lambda}_{i})\left( \overline{\Lambda}_i+\frac{n^{2}}{4}\right),
\end{aligned}
\end{equation}which is given by Cheng and Yang in {\rm \cite{CY1}}.
\end{rem}

\subsection{Geometry of Isoparametric Foliations}\label{subsec6.2}
In recent years, isoparametric theory has remarkable applications in the research of geometry of submanifolds and spectrum analysis.
For the sake of reader's convenience, we recall some fundamental facts about the isoparametric hypersurfaces and focal submanifolds. For more information on isoparametric hypersurfaces and focal submanifolds, we refer the readers to the good articles \cite{CR,TY2}. Firstly, let us introduce the definition of isoparametric functions. For this purpose, let $b$ and $a$ be a smooth function and a continuous function defined on $\mathbb{R}$, respectively.  Let $f$ be a smooth function defined on $\mathbb{S}^{n+1}(1)$.  If $f$ satisfies

\begin{equation}\label{transnormal}
|\nabla f|_{g}^{2}=b(f),\end{equation}and

\begin{equation}\label{iso-ineq-1}
\Delta f=a(f),
\end{equation}
then it is said to be isoparametric.  A function satisfying only
\eqref{transnormal} is called transnormal. The geometric meaning of \eqref{transnormal} and \eqref{iso-ineq-1} is that the regular level
hypersurfaces of $f$ are parallel with each other and have constant mean curvatures. In this sense,
the regular level hypersurfaces of $f$ are called \emph{isoparametric hypersurfaces}, and the two singular level
sets of $f$ are called \emph{focal submanifolds}. Of course, one can also define isoparametric hypersurfaces via an extrinsic geometric viewpoint as follows: A hypersurface $\mathcal{M}^{n}$ in the $(n+1)$-dimensional unit sphere $\mathbb{S}^{n+1}(1)$ is
said to be isoparametric, if all of the principle curvatures  are constant functions.  A well-known result of Cartan states that isoparametric hypersurfaces come as a family of parallel hypersurfaces. To be more specific, given an isoparametric hypersurface $\mathcal{M}^{n}$ in $\mathbb{S}^{n+1}(1)$ and a smooth field $\xi$ with unit normals to $\mathcal{M}^{n}$, for each $x \in \mathcal{M}^{n}$ and $\theta \in \mathbb{R} ,$ we can define $\phi_{\theta}: \mathcal{M}^{n} \rightarrow$ $\mathbb{S}^{n+1}(1)$ by
$$
\phi_{\theta}(x)=\cos \theta x+\sin \theta \xi(x).
$$
Here, $\phi_{\theta}(x)$ is the point at an oriented distance $\theta$ to $\mathcal{M}$ along the normal geodesic through $x$. If $\theta \neq \theta_{\epsilon}$ for any $\epsilon=1, \cdots, \ell$, where $\ell$ denotes the number of distinct constant principal curvatures, $\phi_{\theta}$ is a parallel hypersurface to $M$ at an oriented distance $\theta$. If $\theta=\theta_{\epsilon}$ for some $\epsilon=1, \cdots,\ell,$ it is easy to find that for any vector $Y$ in the principal distributions $$E_{\epsilon}(x)= \left\{Y \in T_{x} M \mid \mathcal{A}_{\xi} Y=\cot \theta_{\epsilon}Y\right\},$$ where $\mathcal{A}_{\xi}$ is a shape operator with
respect to $\xi$,

$$\left(\phi_{\theta}\right)_{*}Y=\left(\cos \theta-\sin \theta \cot \theta_{\epsilon}\right) Y=\frac{\sin \left(\theta_{\epsilon}-\theta\right)}{\sin \theta_{\epsilon}}Y=0 .$$ In
other words, if $\cot \theta=\cot \theta_{\epsilon}$ is a principal curvature of $\mathcal{M}^{n}$, $\phi_{\theta}$ is not an immersion, but is actually a focal submanifold of codimension $m_{\epsilon}+1$  in $\mathbb{S}^{n+1}(1)$.  Using an elegant topological method, M\"{u}nzner proved the remarkable result that the number $\ell$ must be $1,2,3,4,$ or 6 $m_{\epsilon}=m_{\epsilon+2}($ indices mod $\ell) ; \theta_{\epsilon}=\theta_{1}+\frac{\epsilon-1}{\ell} \pi(\epsilon=1, \cdots, \ell) ;$ and when
$\ell$ is odd, $m_{1}=m_{2}$(cf. \cite{Mun1,Mun2}).
M\"{u}nzner asserted that regardless of the number of distinct principal curvatures of $M,$ there are only two distinct focal submanifolds in a parallel family of isoparametric hypersurfaces, and every isoparametric hypersurface is a tube of constant radius over each focal submanifold. We denote the distinct focal submanifolds by $\mathcal{M}_{1}, \mathcal{M}_{2}$ according to the inverse
images of maximum or minimum values of $f$ satisfy the equations system \eqref{transnormal} and \eqref{iso-ineq-1}, respectively. It is well known that $\mathcal{M}_{i}$, where $i=1,2$, are
minimal submanifolds in $\mathbb{S}^{n+1}(1)$.
Assuming that $\left\{\textbf{P}_{0}, \cdots, \textbf{P}_{m}\right\}$ is a symmetric Clifford system on $\mathbb{R}^{2 l}$, this is, $\textbf{P}_{i}$ 's are symmetric matrices satisfying $$\textbf{P}_{i} \textbf{P}_{j}+\textbf{P}_{j} \textbf{P}_{i}=2 \delta_{i j} \textbf{I}_{2 l},$$ in \cite{FKM}, Ferus, Karcher and M\"{u}nzner constructed a polynomial function $\Re$ on $\mathbb{R} ^{2l}$ as follows:
$$
\begin{array}{c}
\Re: \mathbb{R} ^{2 l} \rightarrow \mathbb{R}, \\
\Re(x)=|x|^{4}-2 \sum_{i=0}^{m}\left\langle \textbf{P}_{i} x, x\right\rangle^{2}.
\end{array}
$$
Then, each level hypersurface of $f=\left.\Re\right|_{S^{2 l-1}}$, i.e., the preimage of some regular value of $f,$ is an isoparametric hypersurface with four distinct constant principal curvatures. We choose $\xi=\frac{\nabla f}{|\nabla f|},$  and it can be asserted that $\mathcal{M}_{1}=f^{-1}(1), \mathcal{M}_{2}=f^{-1}(-1)$, with codimensions $m_{1}+1$ and $m_{2}+1$ in $\mathbb{S}^{n+1}(1),$ respectively. The multiplicity pairs $\left(m_{1}, m_{2}\right)=(m, l-m-1),$ provided $m>0$ and $l-m-1>0,$ where $l=k \delta(m)$ $(k=1,2,3, \cdots)$ and $\delta(m)$ is the dimension of an irreducible module of the Clifford algebra $C _{m-1}$ on $\mathbb{R} ^{l}$. See \cite{GTY}.

\subsection{Eigenvalues on the Isoparametric Hypersurfaces of Laplacian}\label{subsec6.3}

In 1982, Yau posed a famous conjecture as follows:
\begin{con}\label{yau-conj}\textbf{\emph{(Yau's Conjecture{\rm \cite{Yau}})}}~  The first nontrivial(non-zero) eigenvalue of Beltrami-Laplacian for every closed embedding minimal hypersurface in the unit sphere equals to the dimension of the hypersurface.\end{con}

Attacking Yau's conjecture, a significant breakthrough to it was made by Choi and Wang \cite{CW}. They proved that the first eigenvalue of every (embedded ) closed minimal hypersurface in $\mathbb{S}^{n+1}(1)$ is not smaller than $\frac{n}{2}$. Furthermore, Brendle pointed out that the first eigenvalue is larger than $\frac{n}{2}$ in his survey paper \cite{Br}. Usually, the calculation of the eigenvalues of the Beltrami-Laplacian, even of the first eigenvalue, is rather complicated and difficult. Up to now, Yau's conjecture remains unsolved.
In 2013, Tang and Yan made an extremely important contribution to this conjecture in \cite{TY1}, where they made an affirmative answer to this conjecture under the condition that $\mathcal{M}^{n}$ is a closed embedding isoparametric hypersurfaces  in $\mathbb{S}^{n+1}(1)$. For more progress on this conjecture, we refer the readers to \cite{Ko,MOU,Mut,Sol1,Sol2,Sol3} and references therein. As a fascinating application of Theorem \ref{thm7.1}, we can show the following result.

\begin{thm}\label{thm7.7}Let $\mathcal{M}^{n}$ be an $n$-dimensional compact minimal isoparametric hypersurface in the unit sphere $\mathbb{S} ^{n+1}(1)$. Then, eigenvalues of the closed eigenvalue problem \eqref{closed-prob} of the Beltrami-Laplacian satisfy
\begin{equation}\label{iso-ineq}
\frac{1}{n} \sum_{k=1}^{n} \overline{\Lambda}_{n_{0}+k} \leq 2 n+4,
\end{equation}where $n_{0}$ denotes the value of the multiplicity of the first eigenvalue.\end{thm}

\begin{proof}
Assume that $\mathcal{M}^{n}$ is a unit sphere $\mathbb{S} ^{n+1}(1)$, the assertion is obvious. Now, we consider that $M^{n}$ is a minimal isoparametric hypersurface other than $\mathbb{S} ^{n}(1),$ we know that
$\overline{\Lambda}_{1}=\overline{\Lambda}_{2}=\cdots=\overline{\Lambda}_{n_{0}}=n$ according to some results showed by Tang and Yan in \cite{TY1}. From   \eqref{z-ineq-22},  we directly get \eqref{iso-ineq}.
\end{proof}

\begin{rem}Let $\mathcal{M}^{n}$ be an $n$-dimensional compact minimal isoparametric hypersurface in the unit sphere $\mathbb{S} ^{n+1}(1)$ and $|\nu|_{g_{0}}=0$. Then, according to Theorem {\rm \ref{thm7.7}}, we get an estimate for the upper bound of the second non-zero eigenvalue without counting the multiplicities of eigenvalues as follows: \begin{equation}\label{iso-ineq-new}
\overline{\Gamma}_{2} \leq 2 n+4.
\end{equation} \end{rem}

\begin{rem}Let $\mathcal{M}^{n}$ be an $n$-dimensional   unit sphere $\mathbb{S} ^{n}(1)$ and $|\nu|_{g_{0}}=0$. Then,  we have
$$
\overline{\Gamma}_{2}=2n+2,
$$ which means that eigenvalue inequality given in Theorem {\rm \ref{iso-ineq}} is very sharp. \end{rem}

\begin{rem}Let $\mathcal{M}^{n}$ be an $n$-dimensional compact minimal isoparametric hypersurface in the unit sphere $\mathbb{S} ^{n+1}(1)$ and $|\nu|_{g_{0}}=0$.  Then, from
\eqref{yang-closed}, we can obtain a weaker inequality than \eqref{iso-ineq}. To be more specific, we have
$$\overline{\Gamma}_{2}\leq2n+4.$$

\end{rem}

\begin{rem} \label{Rem-Sol3}In {\rm \cite{Sol3}},
Solomon  constructed an eigenfunction on a so-called quartic isoparametric hypersurface
$\mathcal{M}^{n}$ of OT-FKM-type, to conclude that $\mathcal{M}^{n}$ has $2n$ as an eigenvalue to fill the gap of eigenvalue sequence $0,n,3n,4n,\cdots$, which contain in the spectrum of Laplacian on the quartic isoparametric hypersurface
$\mathcal{M}^{n}$. Therefore, Theorem {\rm \ref{thm7.7}} further hints that $2n$ could be the second non-zero eigenvalue, although we still don't know whether $2n$ is the  second non-zero eigenvalue or not.
Furthermore, for the isoparametric hypersurfaces of OT-FKM type, Tang and Yan revealed an important fact that $2n$ is an eigenvalue of Beltrami-Laplacian in  {\rm \cite{TY3}}, whose eigenfunction is an isoparametric function. We remind that there is still a question: is it true that $2n$ is the second eigenvalue? In fact, it is an extremely difficult problem, and up to now, it remains open. However, Theorem {\rm \ref{thm7.7}} further indicates that  $2n$  could be the second eigenvalue of Beltrami-Laplacian in the isoparametric minimal hypersurfaces of OT-FKM type. \end{rem}

\subsection{Eigenvalues on the Focal Submanifolds of Laplacian}\label{subsec6.4}In this subsection, we are concerned with the focal submanifolds. It is remarkable that the focal submanifolds of isoparametric hypersurfaces
provide infinitely many spherical submanifolds with abundant intrinsic and extrinsic geometric properties.
For instance, they are both minimal in a unit sphere. Moreover, two focal submanifolds of an isoparametric
hypersurface with four distinct principal curvatures are both Willmore in a unit sphere. See \cite{LY}.
Let $\mathcal{M}_{1}$ be the focal submanifold of an isoparametric hypersurface with four distinct principal curvatures in the unit sphere
$\mathbb{S}^{n+1}(1)$ with codimension $m_{1}+1$.  Tang and Yan {\rm \cite{TY1}} investigated the eigenvalue of Laplacian on the focal submanifold of an isoparametric hypersurface with four distinct principal curvatures and obtained an estimates for the lower bound as follows:

\begin{equation*}\Lambda_{n+3}\left(\mathcal{M}_{1}\right)\geq \frac{4(n+2)\left(m_{2}-1\right)}{n},
\end{equation*}which implies that

\begin{equation*}\overline{ \Gamma}_{2}\left(\mathcal{M}_{1}\right)\geq \frac{4(n+2)\left(m_{2}-1\right)}{n}.\end{equation*} Applying Theorem \ref{thm7.1}, we can get an estimates for the upper bound of the eigenvalues Laplacian on the  focal submanifold of an isoparametric hypersurface with four distinct principal curvatures. This is what the following theorem states.

\begin{thm} \label{thm-6.21} Let $\mathcal{M}_{1}$ be the focal submanifold of an isoparametric hypersurface with four distinct principal curvatures with  dimension   $$\operatorname{dim} \mathcal{M}_{1} \geq \frac{2n}{3}+1$$ in the unit sphere
$\mathbb{S}^{n+1}(1)$. Then, for the eigenvalues of Beltrami-Laplacian, we have

\begin{equation}
\begin{aligned}\label{focal-ineq-1}
\frac{1}{m_{1}+2m_{2}}\sum^{m_{1}+2m_{2}}_{k=1}\overline{\Lambda}_{n+2+k} \leq 2(n+m_{2}+2).
\end{aligned}
\end{equation}
In particular, we have

\begin{equation}\label{focal-ineq-2}
  \overline{\Gamma}_{2} \leq 2(n+m_{2}+2).
\end{equation}
A similar conclusion holds for $\mathcal{M}_{2}$ under an analogous condition.\end{thm}
\begin{proof}If $\operatorname{dim} \mathcal{M}_{1} \geq \frac{2}{3} n+1,$ Tang and Yan \cite{TY1} proved that,
$$
\overline{\Lambda}_{1}\left(\mathcal{M}_{1}\right)=m_{1}+2m_{2}
$$
with multiplicity $n+2.$ Therefore, it follows from \eqref{z-ineq-22} that,

\begin{equation}
\begin{aligned}
\sum^{m_{1}+2m_{2}}_{k=1}\overline{\Lambda}_{n+2+k} \leq\left[\left(m_{1}+2m_{2}\right)+4\right]\left(m_{1}+2m_{2}\right) +\left(m_{1}+2m_{2}\right)^{2},
\end{aligned}
\end{equation}which gives \eqref{focal-ineq-1}, since $n=2(m_{1}+m_{2})$. From \eqref{focal-ineq-1}, it is not difficult to conclude \eqref{focal-ineq-2}. This completes the proof of Theorem \ref{thm-6.21}.

\end{proof}

\begin{rem}
Both $\mathcal{M}_{1}$ and $\mathcal{M}_{2}$ are fully embedded in $\mathbb{S}^{n+1}(1)$ if
$\ell \geq 3$, namely, they cannot be embedded into a hypersphere, which means that, the dimension $n-m_{1}$ (resp. $n-m_{2}$)
of $\mathcal{M}_{1}$ is an eigenvalue of $\mathcal{M}_{1}$ (resp. $\mathcal{M}_{2}$) with multiplicity at least $n + 2$ (cf. {\rm \cite{TY1}}).\end{rem}
 For the focal submanifold $\mathcal{\mathcal{M}}_{1}$ of OT-FKM type in $\mathbb{S}^{5}(1)$ with $\left(\mathcal{M}_{1}, \mathcal{M}_{2}\right)=$ (1,1), Tang, Xie and Yan \cite{TXY} proved that
$
\overline{\Lambda}_{1}\left(\mathcal{M}_{1}\right)=\operatorname{dim} M_{1}=3
$
with multiplicity $6$. Furthermore, for the focal submanifold $\mathcal{M}_{1}$ of homogeneous OT-FKM type in $\mathbb{S}^{15}(1)$ with $\left(m_{1}, m_{2}\right)=(4,3)$, they claimed that
$
\overline{\Lambda}_{1}\left(\mathcal{M}_{1}\right)=\operatorname{dim} \mathcal{M}_{1}=10
$
with multiplicity $16$. Thus, we can prove the following theorems in the light of the idea of the proof of Theorem \ref{thm-6.21}.

\begin{thm}For the focal submanifold $\mathcal{M}_{1}$ of OT-FKM type in $\mathbb{S}^{5}(1)$ with $\left(m_{1}, m_{2}\right)=$ (1,1), we have
\begin{equation}
\begin{aligned}\label{focal-ineq-3}
\sum^{3}_{k=1}\overline{\Lambda}_{6+k} \leq 30.
\end{aligned}
\end{equation}
 In particular, for the second eigenvalue (without considering the multiplicity) of the Beltrami-Laplacian, we have

\begin{equation}\label{focal-ineq-4}
  \overline{\Gamma}_{2} \leq 10.\end{equation}\end{thm}

\begin{thm}
For the focal submanifold $\mathcal{M}_{1}$ of homogeneous OT-FKM type in $\mathbb{S}^{15}(1)$ with $\left(m_{1}, m_{2}\right)=(4,3)$, we have
\begin{equation*}
\begin{aligned}
\sum^{10}_{k=1}\overline{\Lambda}_{16+k} \leq 240.
\end{aligned}
\end{equation*}
In particular, we have

\begin{equation*}
  \overline{\Gamma}_{2} \leq 24.\end{equation*}\end{thm}

\begin{rem}
For the focal submanifold $\mathcal{M}_{1}$ of homogeneous OT-FKM type with $\left(m_{1}, m_{2}\right)=(1,k)$, according to the Proposition 1.1 in {\rm \cite{TY1}} and Theorem {\rm \ref{thm7.1}}, we can similarly give an upper estimate for the non-zero eigenvalue without counting the multiplicity of eigenvalues. \end{rem}

\begin{rem}It is well known that, when $\ell=2$, the focal submanifolds are isometric to $\mathbb{S}^{p}(1)$ and $\mathbb{S}^{q}(1)$. Thus, their second non-zero eigenvalues (without counting the multiplications) equal to two times of their dimensions, respectively.\end{rem}
\begin{rem}  When $\ell=3,$  Cartan showed that $m_{1}=m_{2}=1,2,4$ or $8 .$ In the unit spheres $\mathbb{S}^{4}(1), \mathbb{S}^{7}(1), \mathbb{S}^{13}(1)$ and $\mathbb{S}^{25}(1)$, the focal submanifolds of them are the Veronese embedding of $\mathbb{R} P^{2}, \mathbb{C} P^{2}, \mathbb{H} P^{2}$ and $\mathbb{Q} P^{2},$ respectively. The induced metric of this $\mathbb{R} P^{2}$ minimally embedded in $\mathbb{S}^{4}(1)$ differs from the standard metric of constant Gaussian curvature $K=1$ by a constant factor such that $K=\frac{1}{3}$. As for $\mathbb{C} P^{2}, \mathbb{H} P^{2}$ and $\mathbb{Q} P^{2},$ these
are minimally embedded in the unit spheres $\mathbb{S}^{7}(1), \mathbb{S}^{13}(1)$ and $\mathbb{S}^{25}(1)$ respectively, while the induced metric differs from the symmetric space metric by a constant factor such that $\frac{1}{3} \leq \operatorname{Sec} \leq \frac{4}{3}$. According to {\rm \cite{Str,Mas}}, one knows that the first eigenvalues of the focal submanifolds $\mathbb{C} P^{2}, \mathbb{H} P^{2}$ and $\mathbb{Q} P^{2}$ are equal to their dimensions, respectively. In conclusion, when $\ell=3$, one can assert that

\begin{equation*}
\begin{aligned}
\frac{1}{n_{0}}\sum^{n_{0}}_{k=1}\overline{\Lambda}_{m_{0}+k} \leq 2n_{0}+4,
\end{aligned}
\end{equation*}which implies that

\begin{equation*}
\begin{aligned}
\overline{\Gamma}_{2}\leq 2n_{0}+4,
\end{aligned}
\end{equation*}where $n_{0}$ denotes the dimension of focal submanifolds and $m_{0}$ denotes the multiplicity of the first non-zero eigenvalue.
\end{rem}

\section{Some Conjectures and Further Remarks} \label{sec7}
In this section, we raise  some conjectures and give some further remarks to end this paper.

Let $\Omega$ be a bounded domain on an $n$-dimensional Riemannian manifold $\mathcal{M}^{n}$ with piecewise smooth boundary $\partial\Omega$. We consider  Dirichlet eigenvalue problem of  Laplacian on complete Riemannian manifolds as follows:

\begin{equation}\label{diri-prob-la} {\begin{cases} \
\Delta u +\Lambda u=0, \ \ & {\rm in} \ \ \ \ \Omega, \\
  \ u=0, \ \ & {\rm on} \ \ \partial \Omega.
\end{cases}}\end{equation} We suppose that $\Lambda_{k}$ is the $k^{th}$ eigenvalue corresponding to the eigenfunction $u_{k}$. It is well known
that the spectrum of this eigenvalue problem \eqref{diri-prob-la} is real and discrete. Furthermore, the
following Weyl¡¯s asymptotic formula holds (cf. \cite{Cha}):

\begin{equation}\label{asymptotic}
\Lambda_{k} \sim \frac{4 \pi^{2}}{\left(\omega_{n} \operatorname{vol} \Omega\right)^{\frac{2}{n}}} k^{\frac{2}{n}}, \quad k \rightarrow \infty.
\end{equation} From this asymptotic formula \eqref{asymptotic}, it is not difficult to infer that

\begin{equation*}
  \sum_{i=1}^{k} \Lambda_{i} \sim \frac{n}{n+2} \frac{4 \pi^{2}}{\left(\omega_{n} \operatorname{vol} \Omega\right)^{\frac{2}{n}}} k^{\frac{n+2}{n}}, \quad k \rightarrow \infty.
\end{equation*}
In addition, for any positive integer $n_{1}$, it is easy to show that the eigenvalues of the  eigenvalue problem \eqref{diri-prob-la} of Laplacian satisfy:

\begin{equation*}\begin{aligned}    \lim_{j\rightarrow+\infty}\frac{  \Lambda_{j+1}+ \Lambda_{j+2}+\cdots+\Lambda_{j+n_{1}}}{\Lambda_{j}}=n_{1}.
\end{aligned}\end{equation*}In particular, when $n_{1}=n$, we have

\begin{equation}\begin{aligned} \label{co-in-1}  \lim_{j\rightarrow+\infty}\frac{  \Lambda_{j+1}+ \Lambda_{j+2}+\cdots+\Lambda_{j+n}}{\Lambda_{j}}=n.
\end{aligned}\end{equation}
 From \eqref{co-in-1}, we know that \eqref{1.16} can be improved and thus the first author propose the following conjecture.
\begin{con}Let $\Omega$ be a bounded domain with piecewise smooth boundary on an $n$-dimensional Euclidean space $\mathbb{R}^{n}$. Then, the eigenvalues of the  eigenvalue problem \eqref{diri-prob-la} of the Laplace operator satisfy the following universal inequality:

\begin{equation}\begin{aligned} \label{co-in-1}   \frac{  \Lambda_{j+1}+ \Lambda_{j+2}+\cdots+\Lambda_{j+n}}{\Lambda_{j}}\leq \frac{  \Lambda_{2}+ \Lambda_{3}+\cdots+\Lambda_{n+1}}{\Lambda_{2}},
\end{aligned}\end{equation}for any $j=1,2,\cdots.$  \end{con}
\begin{rem}If the conjecture above is true, it is natural for us to ask the same problem for the case of general Riemannian manifolds, too. \end{rem}
Let $\mathfrak{L}_{\nu}$ be an $n$-dimensional compact Riemannian manifolds without boundary. In this section, we shall investigate eigenvalues
of the closed eigenvalue problem of Laplacian on the Riemannian manifolds $\mathcal{M}^{n}$ as follows:
\begin{equation*}
\Delta \overline{u} +\overline{\Lambda} \overline{u}=0, \ \  {\rm in} \ \ \ \ \mathcal{M}^{n}.\end{equation*}

\begin{con}Let $\mathcal{M}^{n}$ be  an $n$-dimensional  compact minimal submanifold in
the unit sphere $\mathbb{S}^{n+p}(1)$. Then, the eigenvalues of the closed eigenvalue problem \eqref{closed-prob-2} of the Beltrami-Laplacian satisfy:

\begin{equation}\begin{aligned} \label{co-in-1}  \sum_{k=1}^{n}\overline{\Lambda}_{j+k} \leq(n + 3)\overline{\Lambda}_{j} +\frac{\overline{\Lambda}_{j}^{2}}{\overline{\Lambda}_{j+1}}+n^{2}.
\end{aligned}\end{equation} \end{con}

\begin{rem}Provided that \eqref{co-in-1} is true and  $\mathcal{M}^{n}$ is a compact minimal isoparametric hypersurface in the unit sphere $\mathbb{S} ^{n+1}(1)$, it is easy to verify  the following inequality

\begin{equation}\label{iso-ineq-n} \overline{\Gamma}_{2}\leq\frac{2n+3+\sqrt{4n^{2}+16n+9}}{2}.\end{equation} Clearly, inequality \eqref{iso-ineq-n} is sharper than inequality \eqref{iso-ineq-new}.
\end{rem}

If $\mathcal{M}^{n}$ is an isoparametric hypersurfaces embedded in the unit sphere $\mathbb{S}^{n+1}(1)$ with $\ell=1$, then $\overline{\Gamma}_{2}=2n+2$.
If $\mathcal{M}^{n}$ is the generalized Clifford torus $\mathbb{S}^{p}\left(\sqrt{\frac{p}{n}}\right) \times \mathbb{S}^{q}\left(\sqrt{\frac{q}{n}}\right)$ $(p+q=n)$, by a straightforward calculation, we can show that the second eigenvalue $\overline{\Gamma}_{2}=2n$. When $\ell=2$, as is well known, the isoparametric hypersurface in $\mathbb{S}^{n+1}(1)$ is isometric to the  Clifford torus. Thus, $\overline{\Gamma}_{2}=2n$ when $\ell=2$.
As a further interest, based on the argument in the previous section, the first author propose the following conjecture, which is closely  related  to Yau's Conjecture:

\begin{con}\label{Z-conj} Let $\mathcal{M}^{n}$ be an  $n$-dimensional closed minimal hypersurface embedded into the $(n+1)$-dimensional unit sphere $\mathbb{S}^{n+1}(1)$. Then, we have$$
2 n \leq  \overline{\Gamma} _{2} \leq 2 n+2.
$$In particular, let $\mathcal{M}^{n}$ be an $n$-dimensional closed minimal isoparametric hypersurface embedded into $\mathbb{S}^{n+1}(1)$. Then, when $\ell=2,3,4,6$, we have $$\overline{\Gamma}_{2}=2n.$$
\end{con}

\begin{rem}Yau's Conjecture is to consider the first non-zero eigenvalue, while Conjecture {\rm \ref{Z-conj}} is to explore the second non-zero eigenvalue. However, we consider the lower and upper bounds for the second eigenvalue without counting the multiplications of eigenvalues in Conjecture {\rm \ref{Z-conj}}. \end{rem}
\begin{rem}If the last part of Conjecture {\rm \ref{Z-conj}} holds, the second eigenvalue of Laplacian will give a perfect and new character for the isoparametric hypersurfaces embedding into the unit sphere $\mathbb{S}^{n+1}(1)$.\end{rem}

Hsiang and Lawson \cite{HL} showed that every homogenous hypersurfaces in the unit sphere $\mathbb{S}^{n+1}(1)$
is represented as an orbit of a linear isotropy group of a Riemannian symmetric space of rank $2$. We refer the readers to \cite{MO} for the list of the complete examples of
homogenous hypersurfaces in the unit sphere  $\mathbb{S}^{n+1}(1)$. In what follows, there are some further remarks on the eigenvalues of  Beltrami-Laplacian on the isoparametric hypersurfaces.

\begin{rem}  It is well known that, both $SO(3)/(\mathbb{Z}_{2}+\mathbb{Z}_{2})$ and $SU(3)/\mathbb{T}^{2}$ are two isoparametric hypersurfaces embedded in the unit sphere $\mathbb{S}^{n+1}(1)$ with $\ell=3$, and from {\rm \cite{MOU}}, we know that  Conjecture {\rm \ref{Z-conj}} holds   for the cases of $SO(3)/(\mathbb{Z}_{2}+\mathbb{Z}_{2})$ and $SU(3)/\mathbb{T}^{2}$.
\end{rem}

\begin{rem}   Assume that $\mathcal{M}^{n}$ are the cubic isoparametric minimal hypersurfaces  with $n=3,6,12,24$, or $\mathcal{M}^{n}$,  Solomon proved that {\rm\cite{Sol1,Sol2}}, without considering the multiplicity of eigenvalues, $2n$ is an eigenvalue filling an obvious eigenvalue gap: $0,n,3n$. Likewise, Solomon considered  a class of focal submanifolds with  quadratic forms in the quartic isoparametric hypersurfaces and an analogous results obtained in {\rm\cite{Sol3}}, or see {\rm \cite{TY3}}. However, Solomon has not verified whether $2n$ is the second non-zero eigenvalue or not for those isoparametric hypersurfaces. Just like the case of the first eigenvalue, it shall be non trivial to prove that $2n$ is the second eigenvalue.
\end{rem}

The following is a famous conjecture proposed by Chern(cf. \cite{SWY,Yau}):

  \begin{con}\label{chern-conj}\textbf{\emph{(Chern Conjecture)}}~A closed, minimally immersed hypersurface in  the $(n+1)$-dimensional unit sphere $\mathbb{S}^{n+1}(1)$,
whose scalar curvature is a constant, is isoparametric.\end{con}
Furthermore, suppose that $M^{n}$ has constant scalar curvature, and the first author raise a conjecture as follows:

\begin{con}\label{regidity-conj}\textbf{\emph{(Rigidity Conjecture)}}~Let $\mathcal{M}^{n}$ be an  $n$-dimensional closed minimal hypersurface embedded into the $(n+1)$-dimensional unit sphere $\mathbb{S}^{n+1}(1)$ with constant scalar curvature. Then, the second eigenvalue either satisfies $
\overline{\Gamma} _{2}=2 n $, or $\overline{\Gamma} _{2}=2 n+2
$.
\end{con}
\begin{rem} Essentially, Conjecture {\rm \ref{Z-conj}} and Chern's conjecture {\rm \ref{chern-conj}} imply the rigidity conjecture {\rm \ref{regidity-conj}}. Therefore, form the perspective of  spectrum theory,  it is  also a fantabulous understanding for the isoparametric hypersurfaces if Conjecture {\rm \ref{regidity-conj}} is settled. \end{rem}

\begin{ack}
 The first author expresses his gratitude to professor Q.-M. Cheng for his continuous encouragement and useful discussion in early years. The authors also are debt to professor Mark S. Ashbaugh for sharing his literature {\rm \cite{Ash}}. The research was partially supported by the National Natural
Science Foundation of China  (Grant Nos. 11861036 and 11826213) and the Natural Science Foundation of Jiangxi Province (Grant No. 20171ACB21023). \end{ack}

\end{document}